\documentclass[10pt,reqno]{amsart}
\usepackage{graphicx, amssymb}
\numberwithin{equation}{section}

\newtheorem{Thm}{Theorem}[section]

\newtheorem{Prop}[Thm]{Proposition}

\newtheorem{Lem}[Thm]{Lemma}

\def\bC {\mathbf{C}}

\def\bR {\mathbf{R}}
\def\bS {\mathbf{S}}

\def\bp {\mathbf{p}}
\def\bP {\mathbf{P}}

\def\fH {\mathfrak{H}}
\def\fR {\mathfrak{R}}

\def\cA {\mathcal{A}}
\def\cB {\mathcal{B}}
\def\cC {\mathcal{C}}

\def\cF {\mathcal{F}}
\def\cG {\mathcal{G}}
\def\cH {\mathcal{H}}
\def\cI {\mathcal{I}}
\def\cK {\mathcal{K}}
\def\cL {\mathcal{L}}
\def\cM {\mathcal{M}}

\def\cQ {\mathcal{Q}}
\def\cR {\mathcal{R}}
\def\cS {\mathcal{S}}
\def\cT {\mathcal{T}}

\def\cX {\mathcal{X}}

\def\a {{\alpha}}
\def\b {{\beta}}
\def\g {{\gamma}}
\def\de {{\delta}}
\def\eps {{\varepsilon}}
\def\th {{\theta}}
\def\ka{{\kappa}}
\def\l {{\lambda}}

\def\om {{\omega}}

\def\rstr {{\big |}}
\def\indc {{\bf 1}}

\def\la {\langle}
\def\ra {\rangle}

\def\d {{\partial}}

\newcommand{\Dom}{\operatorname{Dom}}

\newcommand{\supess}{\mathop{\mathrm{ess\,sup}}}

\newcommand{\Ker}{\operatorname{Ker}}
\newcommand{\IM}{\operatorname{Im}}

\newcommand{\be}{\begin{equation}}
\newcommand{\ee}{\end{equation}}
\newcommand{\bmat}{\begin{matrix}}
\newcommand{\emat}{\end{matrix}}
\newcommand{\ba}{\begin{aligned}}
\newcommand{\ea}{\end{aligned}}
\newcommand{\lb}{\label}

\begin{document}

\title[Boundary Layer with Phase Transition in Kinetic Theory]{On the Boundary Layer Equations\\ with Phase Transition\\ in the Kinetic Theory of Gases}

\author[N. Bernhoff]{Niclas Bernhoff}
\address[N.B.]{Department of Mathematics and Computer Science, Karlstad University, 651 88 Karlstad, Sweden}
\email{niclas.bernhoff@kau.se}

\author[F. Golse]{Fran\c cois Golse}
\address[F.G.]{CMLS, \'Ecole polytechnique, 91128 Palaiseau Cedex, France}
\email{francois.golse@polytechnique.edu}

\begin{abstract}
Consider the steady Boltzmann equation with slab symmetry for a monatomic, hard sphere gas in a half space. At the boundary of the half space, it is assumed that the gas is in contact with its condensed phase. The present paper 
discusses the existence and uniqueness of a uniformly decaying boundary layer type solution of the Boltzmann equation in this situation, in the vicinity of the Maxwellian equilibrium with zero bulk velocity, with the same temperature as that
of the condensed phase, and whose pressure is the saturating vapor pressure at the temperature of the interface. This problem has been extensively studied first by Y. Sone, K. Aoki and their collaborators, by means of careful numerical
simulations. See section 2 of  [C. Bardos, F. Golse, Y. Sone: J. Stat. Phys.  \textbf{124} (2006), 275--300] for a very detailed presentation of these works. More recently T.-P. Liu and S.-H. Yu [Arch. Rational Mech. Anal. \textbf{209} (2013), 
869--997] have proposed an extensive mathematical strategy to handle the problems studied numerically by Y. Sone, K. Aoki and their group. The present paper offers an alternative, possibly simpler proof of one of the results discussed in 
[T.P. Liu, S.-H. Yu, loc. cit.]
\end{abstract}

\date{\today}

\subjclass{82C40 76P05 35Q20 (34K18,82B26)}

\keywords{Boltzmann equation, Half-space problem, Knudsen layer, Evaporation and condensation, Generalized eigenvalue problem}
\maketitle


\section{Introduction and Notations}


The half-space problem for the steady Boltzmann equation is to find solutions $F\equiv F(x,v)$ to the Boltzmann equation in the half-space with slab symmetry --- meaning that $F$ depends on one space variable only, henceforth denoted by $x>0$, 
and on three velocity variables $v=(v_1,v_2,v_3)$ --- converging to some Maxwellian equilibrium as $x\to+\infty$. Physically, $F(x,v)$ represents the velocity distribution function of the molecules of a monatomic gas located at the distance $x$ of some given 
plane surface, with velocity $v\in\bR^3$.

Assuming for instance that $v_1$ is the coordinate of the velocity $v$ in the $x$ direction, this half-space problem is put in the form
\be
\lb{HalfSp-F}
\left\{
\ba
{}&v_1\d_xF(x,v)=\cB(F,F)(x,v)\,,\quad v\in\bR^3\,,\,\,x>0\,,
\\
&F(x,v)\to\cM_{1,u,1}(v)\,\,\hbox{ as }x\to+\infty\,.
\ea
\right.
\ee
The Boltzmann collision integral is defined as
$$
\cB(F,F)(x,v):=\iint_{\bR^3\times\bS^2}(F(x,v')F(x,v'_*)-F(x,v)F(x,v_*))|(v-v_*)\cdot\om|d\om dv_*\,,
$$
where $v'$ and $v'_*$ are given in terms of $v$, $v_*$ and $\om$ by the formulas
$$
\ba
v':=&v-((v-v_*)\cdot\om)\om\,,
\\
v'_*:=&v_*+((v-v_*)\cdot\om)\om\,.
\ea
$$
For the moment, we assume that $F$ is, say, continuous in $x$ and rapidly decaying in $v$ as $|v|\to+\infty$, so that the collision integral --- and all its variants considered below --- make sense.

The quadratic collision integral above is polarized so as to define a symmetric bilinear operator as follows:
$$
\cB(F,G):=\tfrac12(\cB(F+G,F+G)-\cB(F,F)-\cB(G,G))\,.
$$
An important property of the Boltzmann collision integral is that it satisfies the the conservation of mass, momentum and energy, i.e. the identities
\be\lb{CollInvB}
\int_{\bR^3}\left(\begin{matrix} 1 \\ v_1\\ v_2\\ v_3\\ |v|^2\end{matrix}\right)\cB(F,G)(v)dv=0
\ee
for all rapidly decaying, continuous functions  $F,G$ defined on $\bR^3$ --- see \S 3.1 in \cite{CerIllPul}.

The notation for Maxwellian equilibrium densities is as follows:
$$
\cM_{\rho,u,\th}(v):=\frac{\rho}{(2\pi\th)^{3/2}}e^{-\frac{(v_1-u)^2+v_2^2+v_3^2}{2\th}}\,.
$$
In the sequel, a special role is played by the centered, reduced Gaussian density $\cM_{1,0,1}$, henceforth abbreviated as
$$
M:=\cM_{1,0,1}\,.
$$

We recall that the Boltzmann collision integral vanishes identically on Maxwellian distributions --- see \S 3.2 in \cite{CerIllPul}):
$$
\cB(\cM_{\rho,u,\th},\cM_{\rho,u,\th})=0\qquad\text{ for all }\rho,\th>0\hbox{ and }u\in\bR\,.
$$
With the substitution 
\be\lb{Def-xi}
\xi=v-(u,0,0)\,,
\ee
on account of the identity $\cB(M,M)=0$, the problem (\ref{HalfSp-F}) is put in the form
\be
\lb{HalfSp-f}
\left\{
\ba
{}&(\xi_1+u)\d_xf(x,\xi)+\cL f(x,\xi)=\cQ(f,f)(x,\xi)\,,\quad\xi\in\bR^3\,,\,\,x>0\,,
\\
&f(x,\xi)\to 0\,\,\hbox{ as }x\to+\infty\,,
\ea
\right.
\ee
where $f$ is defined by the identity
$$
F(x,v)=M(1+f)(x,v-(u,0,0))\,,
$$
while
\be\lb{DefLQ}
\cL f:=-2M^{-1}\cB(M,Mf)\,,\quad\cQ(f,f)=M^{-1}\cB(Mf,Mf)\,.
\ee

As a consequence of (\ref{CollInvB}) 
$$
\int_{\bR^3}\left(\begin{matrix} 1\\ \xi_1\\ \xi_2\\ \xi_3\\ |\xi|^2\end{matrix}\right)\cL f(\xi)Md\xi
=
\int_{\bR^3}\left(\begin{matrix} 1 \\ \xi_1\\ \xi_2\\ \xi_3\\ |\xi|^2\end{matrix}\right)\cQ(f,f)(\xi)Md\xi=0
$$
for all rapidly decaying, continuous functions  $f$ defined on $\bR^3$.

Now, for each $\cR\in O_3(\bR)$ (the group of orthogonal matrices with $3$ rows and columns), one has 
$$
\cB(F\circ\cR,G\circ\cR)=\cB(F,G)\circ\cR
$$
so that
\be\lb{InvO3L}
\cL(f\circ\cR)=(\cL f)\circ\cR\,,\quad\cQ(f\circ\cR,f\circ\cR)=\cQ(f,f)\circ\cR\,,
\ee
for all continuous on $\bR^3$, rapidly decaying functions $F,G,f$. (See \S 2.2.3 in \cite{BouGolPul} for a quick proof of these invariance results.)

Assume that the problem (\ref{HalfSp-f}) with boundary condition
\be
\lb{BCond-f}
f(0,\xi)=f_b(\xi)\,,\qquad\xi_1+u>0
\ee
has a unique solution $f$ in some class of functions that is invariant under the action of $O_3(\bR)$ on the velocity variable $\xi$ (such as, for instance, $L^\infty(\bR_+\times\bR^3;Md\xi dx)$. If 
$$
f_b(\xi_1,\xi_2,\xi_3)=f_b(\xi_1,-\xi_2,-\xi_3)\quad\text{ for all }\xi_2,\xi_3\in\bR\text{ and all }\xi_1>-u\,,
$$
then $f\circ\cR$ is also a solution of (\ref{HalfSp-f})-(\ref{BCond-f}), where
\be\lb{DefR}
\cR:=\left(\begin{matrix}1 &0 &0\\ 0 &-1 &0\\ 0 &0 &-1\end{matrix}\right)
\ee
so that, by uniqueness, $f\circ\cR=f$. Henceforth, we restrict our attention to solutions of (\ref{HalfSp-f}) that are even in $(\xi_2,\xi_3)$, and define
$$
\fH:=\{\phi\in L^2(Mdv)\,|\,\phi\circ\cR=\phi\}\,,\qquad\text{ where }\cR\text{ is defined in \eqref{DefR}.}
$$

We recall that $\cL$ is an unbounded, nonnegative self-adjoint Fredholm operator on $L^2(\bR^3;Md\xi )$ with domain
$$
\Dom(\cL):=\{\phi\in L^2(\bR^3;Md\xi )\,|\,\nu\phi\in L^2(\bR^3;Md\xi )\}\,,
$$
(see Theorem 7.2.1 in \cite{CerIllPul}), where $\nu$ is the collision frequency defined as
$$
\nu(|\xi|):=\iint_{\bR^3\times\bS^2}|(\xi-\xi_*)\cdot\om|M_*d\xi_*d\om\,.
$$
The collision frequency satisfies the inequalities
\be\lb{CollFreqIneq}
\nu_-(1+|\xi|)\le\nu(|\xi|)\le\nu_+(1+|\xi|)\quad\text{ for all }\xi\in\bR^3\,,
\ee
where $\nu_+>1>\nu_->0$ designate appropriate constants --- see formula (2.13) in \cite{CerIllPul}. More specifically, the linearized collision operator $\cL$ is of the form 
\be\lb{nu-K}
\cL\phi(\xi)=\nu(|\xi|)\phi(\xi)-\cK\phi(\xi)\,,\qquad\phi\in\Dom\cL\,,
\ee
where $\cK$ is an integral operator, whose properties are summarized in the proposition below.

\begin{Prop}\lb{P-PropK}
The linear integral operator $\cK$ is compact on $L^2(\bR^3;Md\xi)$ and satisfies the identity $\cK(\phi\circ\cR)=(\cK\phi)\circ\cR$, where $\cR$ is defined in \eqref{DefR}. With the notation
$$
L^{\infty,s}(\bR^3):=\{\phi\in L^\infty(\bR^3)\,|\,(1+|\xi|)^s\phi\in L^\infty(\bR^3)\}\,,
$$
the linear operator 
$$
M^{1/2}\cK M^{-1/2}:\,\phi\mapsto \sqrt{M}\cK(\phi/\sqrt{M})
$$ 
is bounded from $L^2(\bR^3;dv)$ to $L^{\infty,1/2}(\bR^3)$, and, for each $s\ge 0$, from $L^{\infty,s}(\bR^3)$ to $L^{\infty,s+1}(\bR^3)$.
\end{Prop} 

These results are stated as Theorem 7.2.4 in \cite{CerIllPul}, to which we refer for a proof. That $\cK$ is compact in $L^2(\bR^3;Md\xi)$ was proved by Hilbert in 1912; that the twisted operator $M^{1/2}\cK M^{-1/2}$ is bounded from 
$L^2(\bR^3;d\xi)$ to $L^\infty(\bR^3)$ and from $L^\infty_s(\bR^3)$ to $L^\infty_{s+1}(\bR^3)$ was proved by Grad in 1962.

In fact Proposition \ref{P-PropK} is a consequence of the following lemma, which will be needed later.

\begin{Lem}\lb{L-PropK}
The linear integral operator $\cK$ can be decomposed as 
$$
\cK=\cK_1+\cK_2-\cK_3\,,
$$
where
$$
\ba
\cK_1\phi(\xi):=\iint_{\bR^3\times\bS^2}\phi(\xi')M(\xi_*)|(\xi-\xi_*)\cdot\om|d\xi_*d\om\,,
\\
\cK_2\phi(\xi):=\iint_{\bR^3\times\bS^2}\phi(\xi'_*)M(\xi_*)|(\xi-\xi_*)\cdot\om|d\xi_*d\om\,,
\\
\cK_3\phi(\xi):=\iint_{\bR^3\times\bS^2}\phi(\xi_*)M(\xi_*)|(\xi-\xi_*)\cdot\om|d\xi_*d\om\,.
\ea
$$
For $j=1,2,3$, the operator $\cK_j$ is compact on $L^2(\bR^3;Md\xi)$ and satisfies the identity 
$$
\cK_j(\phi\circ\cR)=(\cK_j\phi)\circ\cR
$$
where $\cR$ is defined in \eqref{DefR}. Moreover the linear operators 
$$
M^{1/2}\cK_jM^{-1/2}:\,\phi\mapsto \sqrt{M}\cK_j(\phi/\sqrt{M})
$$ 
are bounded from $L^2(\bR^3;d\xi)$ to $L^\infty_{1/2}(\bR^3)$, and, for each $s\ge 0$, from $L^\infty_s(\bR^3)$ to $L^\infty_{s+1}(\bR^3)$ for $j=1,2,3$.
\end{Lem}

Henceforth, we denote 
$$
\la\phi\ra:=\int_{\bR^3}\phi(\xi)M(\xi)d\xi\,.
$$

Furthermore, 
$$
\Ker\cL=\hbox{Span}\{X_+,X_0,X_-,\xi_2,\xi_3\}
$$
(see Theorem 7.2.1 in \cite{CerIllPul}) where 
$$
X_\pm:=\tfrac1{\sqrt{30}}(|\xi|^2\pm\sqrt{15}\xi_1)\,,\qquad X_0:=\tfrac1{\sqrt{10}}(|\xi|^2-5)\,.
$$
The family $(X_+,X_0,X_-,\xi_2,\xi_3)$ is orthonormal in $L^2(\bR^3;Md\xi)$, and orthogonal for the bilinear form $(f,g)\mapsto\la\xi_1fg\ra$ --- see
\cite{CorGolSul} , with
$$
\la\xi_1X_\pm^2\ra=\pm c\,,\quad\la\xi_1X_0^2\ra=\la\xi_1\xi_2^2\ra=\la\xi_1\xi_3^2\ra=0\,.
$$
Here $c$ is the speed of sound associated to the Maxwellian distribution $M$, i.e.
$$
c:=\sqrt{\tfrac53}\,.
$$

In view of (\ref{InvO3L}), the unbounded operator $\cL$ on $L^2(Mdv)$ induces an unbounded, self-adjoint Fredholm operator on $\fH$ still denoted $\cL$, with domain $\fH\cap\Dom\cL$ and nullspace $\fH\cap\Ker\cL=\hbox{Span}\{X_+,X_0,X_-\}$.


\section{Main Result}


Y. Sone and his collaborators have arrived at the following result by formal asymptotics or numerical experiments \cite{SoneOni, Sone78, SoneAokiYamashi, SoneSugi, AokiSoneYama, AokiNishiSoneSugi, Sone00}. Consider the steady Boltzmann 
equation in (\ref{HalfSp-F}) with boundary conditions
\be
\lb{BCond-F}
\ba
F(0,v)=\cM_{\rho_w,0,T_w}(v)\hbox{ for all }v_1>0\,,\quad F(x,v)\to\cM_{\rho_\infty,u,T_\infty}\hbox{ as }x\to+\infty\,.
\ea
\ee
This boundary condition is relevant in the context of a phase transition in the kinetic theory of gases. In this case, the plane of equation $x=0$ represents the interface separating the liquid phase (confined in the domain $x<0$) from the gaseous 
phase (in the domain $x>0$). The parameter $T_w$ is the temperature at the interface, and $\rho_w$ is the density such that $p_w:=\rho_wT_w$ is the saturation vapor pressure for the gas at the temperature $T_w$, while $T_\infty$ and 
$p_\infty=\rho_\infty T_\infty$ are respectively the temperature and pressure far away from the interface, and  $u$ is the transverse bulk velocity in the gas far away from the interface.

Near $u=0$, the set of parameters $T_\infty/T_w$, $p_\infty/p_w$, and $u$ for which this problem has a solution is as represented in Figure \ref{F-FigSone}. It is a surface for $u<0$ and a curve for $u>0$. The solution $F$ converges exponentially 
fast as $x\to+\infty$; however, the exponential speed of convergence is \textit{not uniform on the surface $S$ as $u\to 0^-$, except on the extension of the curve $C$} on the surface $S$. See section 2 of \cite{BardosFGSone2006}, or chapter 7 of 
\cite{SoneBook2} for a comprehensive review of these numerical results. The role of slowly varying solutions --- i.e. solutions whose exponential decay as $x\to+\infty$ is not unifom as $u\to 0^+$ --- in this problem is explained in detail on pp. 280--282
in \cite{BardosFGSone2006}. The original papers by Y. Sone and his group on this problem can be found in the bibliography of \cite{SoneBook1, BardosFGSone2006, SoneBook2}. Other parts of the set of parameters $T_\infty/T_w$, $p_\infty/p_w$ 
and $u$ for which the half-space problem has a solution than the neighborhood of $(1,1,0)$ represented above have been analyzed in detail in \cite{SoneFGOhwaDoi, Boby, SoneTakaFG}.

In the limit case $u=0$, the only solution is the constant $F=\cM_{1,0,1}$ corresponding with the single point $(1/T_w,-u/c,1/p_w)=(1,0,1)$ on the figure  --- see \cite{BardosFGSone2006} section 5 for a proof.

\begin{figure}\lb{F-FigSone}

\includegraphics[width=12cm]{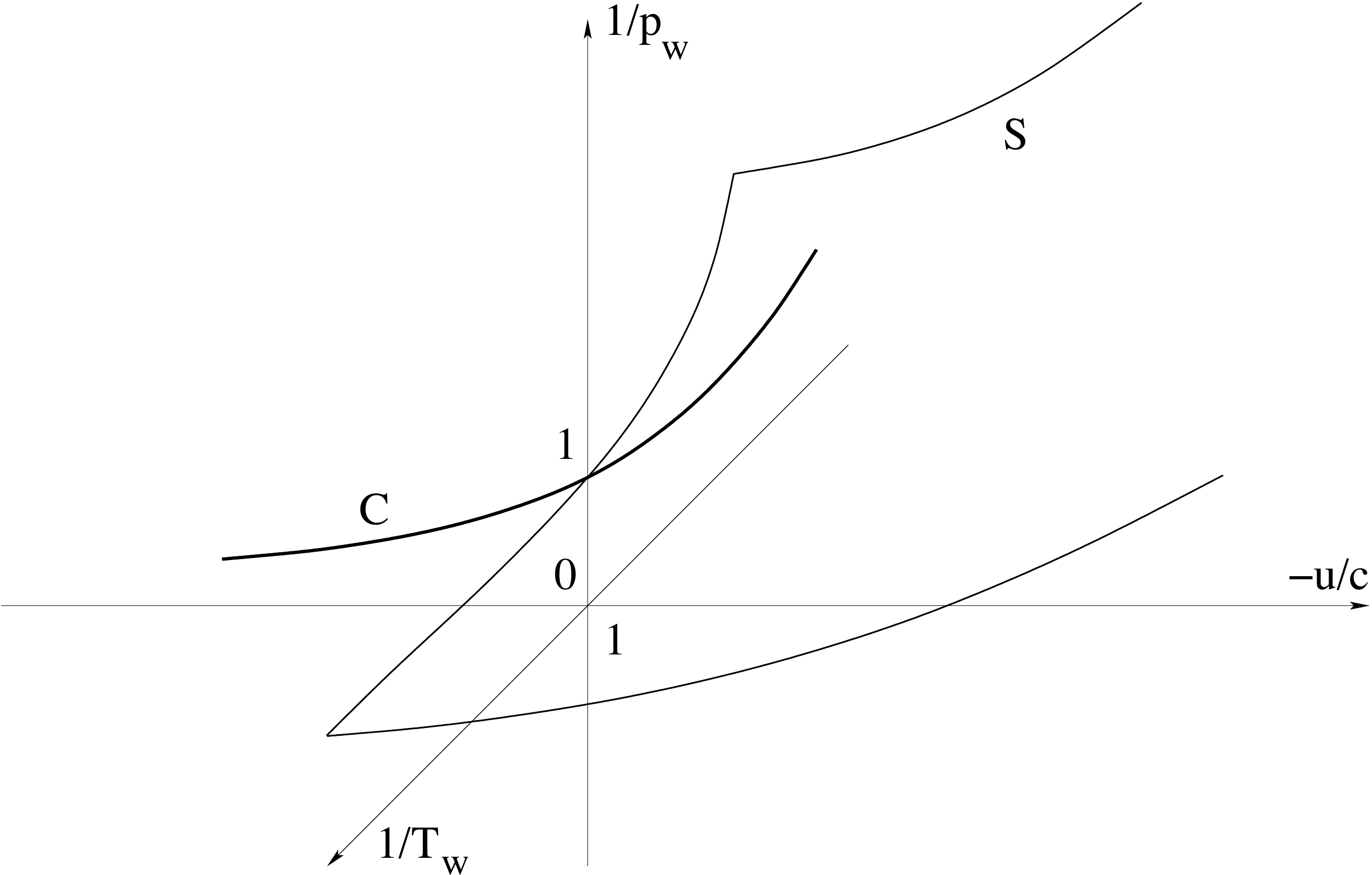}

\caption{The curve $C$ and the surface $S$ in the space of parameters $-u/c$, $p_\infty/p_w$ and $T_\infty/T_w$ near the transition from evaporation
to condensation.}

\end{figure}

We propose a strategy for establishing rigorously the existence of the curve $C$ corresponding with solutions of (\ref{HalfSp-F})-(\ref{BCond-F}) in some neighborhood of the point $(1,0,1)$ converging as $x\to+\infty$ with exponential speed 
\textit{uniformly in $u$}. 

Consider the nonlinear half-space problem for the Boltzmann equation written in terms of the relative fluctuation of distribution function about the normalized Maxwellian $M$
\be\lb{NLHalfSpPbm}
\left\{
\ba
{}&(\xi_1+u)\d_xf_u+\cL f_u=\cQ(f_u,f_u)\,,\quad \xi\in\bR^3\,,\,\,x>0\,,
\\
&f_u(0,\xi)=f_b(\xi)\,,\qquad\xi_1+u>0\,.
\ea
\right.
\ee

\begin{Thm}\lb{T-Main}
There exist $\eps>0$, $E>0$, $R>0$ and $\Gamma>0$ --- defined in \eqref{DefEps}, \eqref{DefE}, \eqref{Cond1} and \eqref{DefGa} respectively --- such that, for each boundary data $f_b\equiv f_b(\xi)$ satisfying
$$
f_b\circ\cR=f_b\quad\text{ and }\,\,\|(1+|\xi|)^3\sqrt{M}f_b\|_{L^\infty(\bR^3)}\le\eps\,,
$$
(with $\cR$ defined in \eqref{DefR}), and for each $u$ satisfying $0<|u|\le R$, the problem \eqref{NLHalfSpPbm} has a unique solution $f_u$ satisfying the symmetry
$$
f_u(x,\cR\xi)=f_u(x,\xi)\quad\text{ for a.e. }\,\,(x,\xi)\in\bR_+\times\bR^3\,,
$$
and the uniform decay estimate 
\be\lb{ExpDecUnif}
\supess_{\xi\in\bR^3}(1+|\xi|)^3\sqrt{M(\xi)}|f_u(x,\xi)|\le Ee^{-\g x}\,,\qquad x>0
\ee
for all $\g$ such that $0<\g<\min(\Gamma,\tfrac12\nu_-)$ if and only if the boundary data $f_b$ satisfies the two additional conditions
\be\lb{2CompCond}
\la(\xi_1+u)Y_1[u]\fR_{u,\g}[f_b]\ra=\la(\xi_1+u)Y_2[u]\fR_{u,\g}[f_b]\ra=0\,.
\ee
The functions $Y_1[u]\equiv Y_1[u](\xi)$ and $Y_2[u]\equiv Y_2[u](\xi)$ are defined in Lemma \ref{L-PenCond}, while the (nonlinear) operator $\fR_{u,\g}$ is defined in \eqref{DefRug}.
\end{Thm}

\bigskip
Several remarks are in order before starting with the proof of Theorem \ref{T-Main}.

\smallskip
First observe that Sone's original problem falls in the range of application of Theorem \ref{T-Main}. Indeed, the boundary condition \eqref{BCond-F} translates into
\be\lb{BCond-f}
f_b(\xi)=\frac{\cM_{\rho_w,-u,T_w}-M}{M}\,,
\ee
which is obviously even in $(\xi_2,\xi_3)$. Since
$$
M^{-1}d\cM_{1,0,1}=\rho_w-1-u\xi_1+(T_w-1)\tfrac12(|\xi|^2-3)\,,
$$
one has
$$
|\rho_w-1|+|u|+|T_w-1|\ll 1\implies\|(1+|\xi|)^3\sqrt{M}f_b\|_{L^\infty(\bR^3)}\ll 1\,.
$$

The two conditions \eqref{2CompCond} are expected to define a ``submanifold of codimension $2$'' in the set of boundary data $f_b$. When specialized to the three dimensional submanifold of Sone's data \eqref{BCond-f}, this ``submanifold 
of codimension $2$'' is expected to be the curve described by the equations
$$
p_\infty/p_w=h_1(u/\sqrt{5/3})\,,\qquad T_\infty/T_w=h_1(u/\sqrt{5/3})\,,
$$
referred to as equations (2.3) in \cite{BardosFGSone2006}, and defining the set of parameters for which a solution of the half-space problem exists in the evaporation case. As explained above, this curves is expected to extend smoothly in the
condensation region if slowly decaying solutions are discarded. Unfortunately, we have not been able to check that the two equations above, even when restricted to the $3$ dimensional manifold of Sone's boundary data \eqref{BCond-f} are 
smooth (at least $C^1$) and locally independent (by the implicit function theorem). We obviously expect this to be true, but this seems to involve some rather delicate properties of half-space problems for the linearized Boltzmann equation.

An a priori estimate to be found in section 5 of \cite{BardosFGSone2006} shows that the only solution of \eqref{NLHalfSpPbm}-\eqref{BCond-f} with $u=0$ is $f_u\equiv 0$, so that $\rho_w=T_w=1$. One can differentiate formally about this point
both sides of the Boltzmann equation at $u=0$ along the curve $u\mapsto(\rho_w(u),T_w(u))$ defined for $-\sqrt{5/3}<u<0$ by the equations (2.3) of \cite{BardosFGSone2006} recalled above. Denoting $\dot f_0(x,\xi):=(\d f_u/\d u)(x,\xi)|_{u=0}$,
one finds that $\dot f_0$ should satisfy
$$
\left\{
\ba
{}&\xi_1\d_x\dot f_0+\cL\dot f_0=0\,,&&\quad \xi\in\bR^3\,,\,\,x>0\,,
\\
&\dot f_0(0,\xi)=\rho_w'(0^-)+T_w'(0^-)\tfrac12(|\xi|^2-3)-\xi_1\,,&&\quad\xi_1+u>0\,,
\\
&\dot f_0(x,\xi)\to 0&&\quad\text{ as }x\to+\infty\,,
\ea
\right.
$$
where $\rho_w'(0^-)$ and $T_w(0^-)$ are the left derivatives of $\rho_w$ and $T_w$ at $u=0$ along the evaporation curve. The Bardos-Caflisch-Nicolaenko \cite{BardCaflNico} theory of the half-space problem for the linearized Boltzmann equation 
implies that there exists a unique pair of real numbers $(\rho_w'(0^-),T_w'(0^-))$ for which a solution $\dot f_0$ exists. This is obviously a very interesting piece of information as it provides a tangent vector at the origin to the ``curve'' defined by the 
two conditions \eqref{2CompCond} of Theorem \ref{T-Main} specialized to boundary data of the form \eqref{BCond-f}. Unfortunately, whether $f_u$ is differentiable in $u$ at $u\not=0$ is rather unclear, and we shall not discuss this issue any further.

Our strategy for proving Theorem \ref{T-Main} is as follows: first we isolate the slowly varying mode near $\rho_w=T_w=1$ and $u=0$ on the condensation side. This leads to a generalized eigenvalue problem of the kind considered by B. Nicolaenko
in \cite{Nico74,Nico76} (see also \cite{NicoThurber, CafNico82}) in his construction of a weak shock profile for the nonlinear Boltzmann equation. Next we remove this slowly varying mode from the linearization of \eqref{NLHalfSpPbm} by a combination
of the Lyapunov-Schmidt procedure used in \cite{Nico74, Nico76, CafNico82} to establish the existence of the shock profile as a bifurcation from the constant sonic Maxwellian, and of the penalization method of \cite{UYY} for studying weakly nonlinear 
half-space problems. Theorem \ref{T-Main} is obtained by a simple fixed point argument about the solution of some conveniently selected linear problem, in whose definition both the Lyapunov-Schmidt method of \cite{Nico74} and the penalization
method of \cite{UYY} play a key role. In some sense, the paper \cite{Golse2008} can be regarded as a precursor to this one; it extends the very clever penalization method of \cite{UYY} to the case $u=0$, but does not consider the transition from $u<0$ 
(evaporation) to $u>0$ (condensation).  We also refer the interested reader to the beginning of section 4, where we explain one (subtle) difference between the results obtained on weakly nonlinear half-space problems for the Boltzmann equation in 
\cite{UYY} and the problem analyzed in the present work.

The outline of the paper is as follows: section 3 provides a self-contained construction of the solution to the Nicolaenko-Thurber generalized eigenvalue problem near $u=0$. Section 4 introduces the penalization method, and formulates the problem
to be solved by a fixed point argument. Section 5 treats the linearized penalized problem, while section 6 treats the (weakly) nonlinear penalized problem by a fixed point argument. Theorem \ref{T-Main} is obtained by removing the penalization.
The main ideas used in the proof of Theorem \ref{T-Main} are to be found in sections 3-4; by contrast, sections 5 and 6 are mostly of a technical nature.

\smallskip
Before starting with the proof of Theorem \ref{T-Main}, we should say that Theorem \ref{T-Main} above is not completely new or original, in the following sense. A general study of the Sone half-space problem with condensation and evaporation
for the Boltzmann equation has been recently proposed by T.-P. Liu and S.-H. Yu in a remarkable paper \cite{LiuYu}. Our Theorem \ref{T-Main} corresponds to cases 2 and 4 in Theorem 28 on p. 984 of \cite{LiuYu}. Given the considerable range of 
cases considered in \cite{LiuYu}, the proof of Theorem 28 is just sketched. The analysis in \cite{LiuYu} appeals to a rather formidable technical apparatus, especially to the definition and structure of the Green function for the linearized Boltzmann
equation (see section 2.2 of \cite{LiuYu}, referring to an earlier detailed study of these functions, cited as ref. 21 in \cite{LiuYu}). Our goal in Theorem \ref{T-Main} is much more modest: to provide a completely self-contained proof for one key item
in the Sone diagram, namely the evaporation and its extension to the condensation regime obtained by discarding slowly decaying solutions. We also achieve much less: for instance we do not know whether the solution $M(1+f_u)$ of the steady 
Boltzmann equation obtained in Theorem \ref{T-Main} satisfies $M(1+f_u)\ge 0$. This is known to be a shortcoming of the method of constructing solutions to the steady Boltzmann equation by some kind of fixed point argument about a uniform 
Maxwellian. At variance, all the results in \cite{LiuYu} are based on an invariant manifold approach based on the large time behavior of the Green function for the linearized Boltzmann equation. (Incidentally, the numerical results obtained by Sone 
and his collaborators were also based on time-marching algorithms in the long time limit.) Since the Boltzmann equation propagates the positivity of its initial data, one way of constructing nonnegative steady solutions of the Boltzmann equation is
to obtain them as the long time limit of some conveniently chosen time-dependent solutions. For this reason alone, the strategy adopted in \cite{LiuYu} has in principle more potential than ours. On the other hand, our proof uses only elementary
techniques, and we hope that the present paper could serve as an introduction to the remarkable series of works by Sone and his collaborators quoted above, and to the deep mathematical analysis in \cite{LiuYu}.


\section{The Nicolaenko-Thurber Generalized Eigenvalue Problem}


The generalized eigenvalue problem considered here is to find $\tau_u\in\bR$ and a generalized eigenfunction $\phi_u\in\fH\cap\Dom\cL$ satisfying
\be
\lb{GEP}
\left\{
\ba
{}&\cL\phi_u=\tau_u(\xi_1+u)\phi_u\,,
\\
&\la(\xi_1+u)\phi_u^2\ra=-u\,,
\ea
\right.
\ee
for each $u\in\bR$ near $0$. 

This problem was considered by Nicolaenko and Thurber in \cite{NicoThurber} for $u$ near $c$ --- see Corollary 3.10 in \cite{NicoThurber}\footnote{The possibility of extending Corollary 3.10 of \cite{NicoThurber} to the case where $|u|\ll 1$ was 
mentioned to the second author by B. Nicolaenko in 1999.}. It is the key to the construction of a weak shock profile for the Boltzmann equation \cite{Nico74,Nico76}. (An approximate variant of (\ref{GEP}) is considered in \cite{CafNico82} for molecular 
interactions softer than hard spheres.)

\begin{Prop}\lb{P-GEP}
There exists $r>0$, a real-analytic function 
$$
(-r,r)\ni u\mapsto\tau_u\in\bR\,,
$$
and a real-analytic map 
$$
(-r,r)\ni u\mapsto\phi_u\in\fH\cap\Dom\cL
$$
that is a solution to (\ref{GEP}) for each $u\in(-r,r)$ and satisfies 
$$
u\tau_u<0\quad\text{ for }0<|u|<r\,.
$$
In other words, 
$$
\tau_0=0\quad\text{ and }\tau_u=u\dot\tau_0+O(u^2)\text{ as }u\to 0\,,
$$
with 
$$
\dot\tau_0<0\,.
$$
Furthermore, there exists a positive constant $C_s$ for each $s\ge 0$ such that $\phi_u$ satisfies 
$$
\|(1+|\xi|)^s\sqrt{M}\phi_u\|_{L^\infty}\le C_s
$$
for all $s\ge 0$, uniformly in $u\in(-r,r)$.
\end{Prop}

Observe that $f_u(x,\xi)=e^{-\tau_ux}\phi_u(\xi)$ is a solution of the steady linearized Boltzmann equation
$$
(\xi_1+u)\d_xf_u+\cL f_u=0\,,\qquad x\in\bR\,,\,\,\xi\in\bR\,.
$$
Since $\tau_u\simeq u\dot\tau_0$ as $u\to 0$ with $\dot\tau_0<0$, one has
$$
\|f_u(x,\cdot)\|_{\fH}=e^{|\dot\tau_0|ux}\|\phi_u\|_{\fH}\qquad\text{ as }x\to+\infty\,.
$$
In other words, $f_u$ grows exponentially fast as $x\to+\infty$ if $u>0$ (evaporation) and decays exponentially fast to $0$ as $x\to+\infty$ for $u<0$ (condensation). In the latter case, the exponential speed of convergence of $f_u$ is 
$|\dot\tau_0||u|$, which is \text{not uniform} as $u\to 0^-$. The transition from the curve $C$ to the surface $S$ when crossing the plane $u=0$ on Figure \ref{F-FigSone} --- which represents the transition from evaporation to condensation 
--- corresponds to the presence of an additional degree of freedom in the set of solutions. At the level of the linearized equation, this additional degree of freedom comes from the mode $f_u(x,\xi)$, which decays to $0$ as $x\to+\infty$, albeit
not uniformly as $u\to 0^-$, if and only if $u<0$. The extension of the curve $C$ on the surface $S$ is defined by the fact that solutions to the boundary layer equation (\ref{HalfSp-f}) decaying exponentially fast as $x\to+\infty$ uniformly in 
$u\to 0^-$ do not contain the $f_u$ mode. 

\smallskip
One can arrive at the statement of Proposition \ref{P-GEP} by adapting the arguments in \cite{NicoThurber} --- especially Theorems 3.7 and 3.9, and Corollaries 3.8 and 3.10, together with Appendices B and D there. Their discussion is based on a 
careful analysis of the zeros of a certain Fredholm determinant --- in fact, of the perturbation of the identity by a certain finite rank operator --- that can be seen as the dispersion relation for the linearized Boltzmann equation. For the sake of being 
self-contained, we give a (perhaps?) more direct, complete proof of Proposition \ref{P-GEP} below. 

\begin{proof}
Consider for each $z\in\bC$ the family of unbounded operators $T(z)=\cL-z\xi_1$ on $\fH$. In view of (\ref{CollFreqIneq}), $T(z)$ is a holomorphic family of unbounded operators with domain $\Dom T(z)=\fH\cap\Dom\cL$ whenever $|z|<\nu_-$, in 
the sense of the definition on p. 366 in \cite{Kato}. (Indeed, defining the operator $U:\,f\mapsto\frac1{1+|\xi|}f$, we see that $U$ is a one-to-one mapping of $\fH$ to $\fH\cap\Dom\cL$ and that  $z\mapsto T(z)U$ is a holomorphic map defined for all 
$z$ such that $|z|<\nu_-$ with values in the algebra of bounded operators on $\fH$.) 

The family $T(z)$ is self-adjoint on $\fH$ in the sense of the definition on p. 386 in \cite{Kato}, since $\cL$ is self-adjoint on $\fH$ and 
$$
T(\overline{z})=\cL-\overline{z}\xi_1=T(z)^*\quad\text{ whenever }|z|<\nu_-\,.
$$
Besides, $\l=0$ is an isolated 3-fold eigenvalue of $T(0)=\cL$, corresponding with the 3-dimensional nullspace $\fH\cap\Ker\cL$ (see Theorem 7.2.5 in \cite{CerIllPul}). As explained on p. 386 in \cite{Kato}, there exist $3$ real-analytic functions 
$z\mapsto\l_+(z),\l_0(z),\l_-(z)$ defined for $z$ real near $0$ and $3$ real-analytic maps $z\mapsto\phi^+_z,\phi^0_z,\phi^-_z$ defined for $z$ real near $0$ with values in $\fH\cap\Dom\cL$ such that, for each real $z$ near $0$,
$$
\l_+(z)\text{ (resp. }\l_0(z),\l_-(z)\text{ ) is an eigenvalue of }T(z)
$$
and 
$$
\ba
(\phi^+_z,\phi^0_z,\phi^-_z)\text{ is an orthonormal system of eigenfunctions of }T(z)\text{ in }\fH
\\
\text{ for the eigenvalues }\l_+(z),\l_0(z),\l_-(z)\text{ respectively}&,
\ea
$$
while 
$$
\l_\pm(0)=0\,,\,\,\l_0(0)=0\,.
$$

For $z=0$, one has
$$
T(0)\phi^\pm_0=T(0)\phi^0_0=0\,,\quad\text{ so that }\,\,\phi^\pm_0\,,\,\,\phi^0_0\in\fH\cap\Ker\cL\,.
$$
Next we differentiate twice in $z$ the identities
$$
T(z)\phi^\pm_z=\l_\pm(z)\phi^\pm_z\quad\text{ and }\quad T(z)\phi^0_z=\l_0(z)\phi^0_z\,.
$$
Denoting by $\dot{}$ the derivation with respect to $z$ and dropping the $\pm$ or $0$ indices (or exponents) for simplicity, we obtain successively
\be\lb{DerVp1}
\cL\dot\phi_z-z\xi_1\dot\phi_z-\xi_1\phi_z=\l(z)\dot\phi_z+\dot\l(z)\phi_z\,,
\ee
and
\be\lb{DerVp2}
\cL\ddot\phi_z-z\xi_1\ddot\phi_z-2\xi_1\dot\phi_z=\l(z)\ddot\phi_z+2\dot\l(z)\dot\phi_z+\ddot\l(z)\phi_z\,.
\ee
Setting $z=0$ in (\ref{DerVp1}) leads to
$$
\cL\dot\phi_0=(\xi_1+\dot\l(0))\phi_0\,.
$$
Since $\phi_0\in\fH\cap\Ker\cL$ and $(\xi_1+\dot\l(0))\phi_0\perp\fH\cap\Ker\cL$, we conclude that 
$$
\dot\l(0)\in\{+c,0,-c\}\,.
$$
(Indeed, the matrix of the quadratic form defined on $\fH\cap\Ker\cL$ by $\phi\mapsto\la(\xi_1+u)\phi^2\ra$ in the basis $\{X_+,X_0,X_-\}$ is
$$
\left(\begin{matrix} u+c &0 &0\\ 0 &u &0\\ 0 &0 &u-c\end{matrix}\right)\,;
$$
this matrix is degenerate if and only if there exists $\phi\in\fH\cap\Ker\cL\setminus\{0\}$ such that $(\xi_1+u)\phi\perp\fH\cap\Ker\cL$, and that happens only if $u=\pm c$ or $u=0$: see \cite{CorGolSul}.)

Furthermore
$$
\left\{
\ba
{}&\dot\l(0)=+c\Rightarrow\phi_0\in\bR X_+\,,
\\
&\dot\l(0)=0\,\,\,\,\Rightarrow\phi_0\in\bR X_0\,,
\\
&\dot\l(0)=-c\Rightarrow\phi_0\in\bR X_-\,,
\ea
\right.
$$
and since $(\phi^+_0,\phi^0_0,\phi^-_0)$ is an orthonormal system in $\fH$, each one of the three cases above occurs for exactly one of the branches $\l_+(z),\l_0(z),\l_-(z)$. 

Henceforth, we label these eigenvalues so that $\dot\l_\pm(0)=\pm c$ and $\dot\l_0(0)=0$ and concentrate on the branch $\l_0(z)$. In particular, up to a change in orientation, one has $\phi^0_0=X_0$ and
\be\lb{DerVp1-0}
\cL\dot\phi^0_0=\xi_1\phi^0_0\,.
\ee

This being done, setting $z=0$ in (\ref{DerVp2}), we arrive at the identity
$$
\cL\ddot\phi^0_0-2\xi_1\dot\phi^0_0=\ddot\l_0(0)\phi^0_0\,.
$$
Taking the inner product of both sides of this identity with $\phi^0_0$, we see that
$$
\la\phi^0_0\cL\ddot\phi^0_0\ra-2\la\xi_1\dot\phi^0_0\phi^0_0\ra=\ddot\l_0(0)\la(\phi^0_0)^2\ra=\ddot\l_0(0)\,.
$$
Since $\phi^0_0=X_0\in\Ker\cL$ and $\cL$ is self-adjoint
$$
\la\phi^0_0\cL\ddot\phi^0_0\ra=0\,.
$$
In view of (\ref{DerVp1-0}), one has
$$
\la\xi_1\dot\phi^0_0\phi_0^0\ra=\la\dot\phi^0_0\cL\dot\phi^0_0\ra>0\,,
$$
since $\dot\phi^0_0\notin\Ker\cL$ --- otherwise $\cL\dot\phi^0_0=\xi_1\phi^0_0=\xi_1X_0=0$ which is obviously impossible. Therefore
\be\lb{ddl(0)}
\ddot\l_0(0)=-2\la\dot\phi^0_0\cL\dot\phi^0_0\ra<0\,.
\ee

To summarize, we have obtained real-analytic maps $z\mapsto\l_0(z)$ and $z\mapsto\phi^0_z$ such that
$$
\l_0(0)=\dot\l_0(0)=0\,,\quad\ddot\l_0(0)<0\,,\quad\text{ and }\phi^0_0=X_0\,,
$$
while
\be\lb{GEPbis}
\cL\phi^0_z=z\xi_1\phi^0_z+\l_0(z)\phi^0_z
\ee
for all $z$ real near $0$. 

Set $u(z):=\l_0(z)/z$; since $\l_0(0)=\dot\l_0(0)=0$ while $\ddot\l_0(0)<0$, the function $u$ is real-analytic near $0$ and satisfies
$$
u(0)=\dot\l_0(0)=0\,,\quad\text{ and }\dot u(0)=\tfrac12\ddot\l_0(0)<0\,.
$$
By the open mapping theorem (see Rudin \cite{RudinRCA}, Theorem 10.32), $z\mapsto u(z)$ extends into a biholomorphic map between two open neighborhoods of the origin that preserves the real axis. Denoting by $u\mapsto z(u)$
its inverse, we see that $\l_0(z(u))=uz(u)$ and we recast (\ref{GEPbis}) in the form
$$
\cL\phi^0_{z(u)}=z(u)\xi_1\phi^0_{z(u)}+uz(u)\phi^0_{z(u)}\,.
$$

For $u$ real sufficiently near $0$, one has 
$$
\nu(|\xi|)-z(u)(\xi_1+u)\ge\tfrac12\nu_-(1+|\xi|)>0\quad\text{ for all }\xi\in\bR^3\,.
$$ 
Then, returning to Hilbert's decomposition (\ref{nu-K}) of the linearized operator $\cL$, we see that
$$
\phi^0_{z(u)}=\frac1{\nu-z(u)(\xi_1+u)}\cK\phi^0_{z(u)}
$$
for all $u$ near $0$. By definition, $\|\phi^0_{z(u)}\|_{\fH}=1$; since $\cK$ is a bounded operator on $\fH$, the identity above implies that\footnote{We denote by $B(X,Y)$ the space of bounded linear operators from the Banach space $X$ to the 
Banach space $Y$, and set $B(X):=B(X,X)$.} 
$$
\|(1+|\xi|)\phi^0_{z(u)}\|_{\fH}\le\tfrac2{\nu_-}\|\cK\|_{B(\fH)}\,,\quad\text{ for all  }u\text{ near }0\,.
$$

By Proposition \ref{P-PropK}, we improve this result and arrive at the bound of the form
$$
\|\sqrt{M}\phi^0_{z(u)}\|_{L^\infty_s}\le C_s
$$
for all $s\ge 0$, uniformly in $u$ near $0$.

Furthermore
$$
\ba
\la(\xi_1+u)(\phi^0_{z(u)})^2\ra&=u\la(\phi^0_0)^2\ra+2z(u)\la\xi_1\dot\phi^0_0\phi^0_0\ra+O(u^2)
\\
&=u+2z(u)\la\dot\phi^0_0\cL\dot\phi^0_0\ra+O(u^2)=u-z(u)\ddot\l_0(0)\,.
\ea
$$
Since $z(u)=2u/\ddot\l_0(0)+O(u^2)$, we conclude that $u\mapsto\la(\xi_1+u)(\phi^0_{z(u)})^2\ra$ is a real-analytic function defined near $u=0$ and satisfying
$$
\la(\xi_1+u)(\phi^0_{z(u)})^2\ra=-u+O(u^2)\,.
$$

Finally, setting $\tau_u:=z(u)$ and
$$
\phi_u:=\frac{\phi^0_{z(u)}}{\sqrt{-\la(\xi_1+u)(\phi^0_{z(u)})^2\ra/u}}\,,
$$
we arrive at the statement of Proposition \ref{P-GEP}.
\end{proof}

\smallskip
\noindent
\textbf{Remarks.}

\noindent
(1) The analogue of Proposition \ref{P-GEP} in the case where $\dot\l(0)=c$ is precisely what is discussed in Corollary 3.10 of \cite{NicoThurber}. The idea of reducing the generalized eigenvalue problem (\ref{GEP}) to a standard eigenvalue 
problem for the self-adjoint family $T(z)$, i.e. of considering $u\tau_u$ as a function of $\tau_u$ near the origin, is somewhat reminiscent of the identity (20) in \cite{NicoThurber}.

\noindent
(2) For inverse power law, cutoff potentials softer than hard spheres, one has 
$$
\nu_-(1+|\xi|)^\a\le\nu(|\xi|)\le \nu_-(1+|\xi|)^\a\quad\text{ for some }\a<1\,.
$$
In that case, the operator $T(z)=\cL-z\xi_1$ is not a holomorphic family on $\fH$, since 
$$
\Dom T(z)=\frac1{(1+|\xi|)^\a(1+|\xi_1|)^{1-\a}}\fH\quad\text{ for }z\not=0\,,
$$
while 
$$
\Dom T(0)=\frac1{(1+|\xi|)^\a}\fH\,.
$$
The argument used in the proof of Proposition \ref{P-GEP} fails for such potentials, which is the reason why Caflisch and Nicolaenko \cite{CafNico82} consider an approximate variant of the generalized eigenvalue problem instead of (\ref{GEP}).


\section{The Penalized Problem}


Our strategy for solving the nonlinear half-space problem (\ref{HalfSp-f}) near $u=0$ --- i.e. near the transition from evaporation to condensation at the interface $x=0$ --- is as follows. 

Consider the nonhomogeneous, linear half-space problem
\be
\lb{HalfSpLin}
\left\{
\ba
{}&(\xi_1+u)\d_xf(x,\xi)+\cL f(x,\xi)=Q\,,\quad\xi\in\bR^3\,,\,\,x>0\,,
\\
&f(0,\xi)=f_b(\xi)\,,\quad\xi_1+u>0\,,
\\
&f(x,\xi)\to 0\,\,\text{ as }x\to+\infty\,,
\ea
\right.
\ee
where 
\be\lb{HypQ}
\left\{
\ba
{}&Q(x,\cdot)\perp\Ker\cL\text{ for each }x>0\,,
\\
&Q(x,\xi_1,\xi_2,\xi_3)=Q(x,\xi_1,-\xi_2,-\xi_3)\,,\text{ for each }x>0\,,\,\,\xi\in\bR^3\,,
\\
&\text{and }Q(x,\xi)\to 0\text{ as }x\to+\infty\,.
\ea
\right.
\ee
All solutions $f$ to this problem considered below are assumed to be even in $(\xi_2,\xi_3)$:
\be\lb{Even-f}
f(x,\xi_1,\xi_2,\xi_3)=f(x,\xi_1,-\xi_2,-\xi_3)\,,\text{ for each }x>0\,,\,\,\xi\in\bR^3\,.
\ee

Assume for now that we can prove existence and uniqueness of a solution $f=\cF_u[f_b,Q]$ to (\ref{HalfSpLin}) provided that $f_b$ and $Q$ satisfy some compatibility conditions, which we denote symbolically as $\cC_u[f_b,Q]=0$. An obvious 
strategy is to seek the solution $f$ of (\ref{HalfSp-f}) with boundary condition (\ref{BCond-f}) as a fixed point of the map $f\mapsto\cF_u[f_b,\cQ(f,f)]$ in some neighborhood of $f=0$.

There are two main difficulties in this approach. First, the nonlinear solution $f$ should satisfy the compatibility conditions $\cC_u[f_b,\cQ(f,f)]=0$; these compatibility conditions are not explicit since they involve $\cQ(f,f)$, and yet satisfying these 
compatibility conditions is necessary in order to be able to define $\cF_u[f_b,\cQ(f,f)]$ in the first place. 

A second difficulty lies with the solution of the linearized problem (\ref{HalfSpLin}) itself. Since $\cQ$ is a quadratic operator, one can indeed expect that the nonlinear operator $f\mapsto\cF_u[f_b,\cQ(f,f)]$ will be a strict contraction in a closed ball 
centered at the origin with small enough positive radius $R_u$, say in some space of the type $M^{-1/2}L^\infty(\bR_+;L^\infty_s(\bR^3))$ for large enough $s$. In other words, solving the linearized problem \eqref{HalfSpLin} in some appropriate
setting is the key step. Once this is done, handling the nonlinearity should not involve intractable, additional difficulties.

In fact, the work of Ukai-Yang-Yu \cite{UYY} solves precisely both these difficulties. Unfortunately, their result is not enough for the purpose of studying the transition from evaporation to condensation, for the following reason. 

Indeed, one faces the following problem: the radius $R_u$ of the closed ball centered at the origin on which one can apply the fixed point theorem to the nonlinear operator $f\mapsto\cF_u[f_b,\cQ(f,f)]$ might be so small that 
$$
f_b(\xi)=\frac{\cM_{\rho,-u,T}-M}{M}\notin\overline{B(0,R_u)}\,.
$$
In other words, 
$$
\|(\cM_{\rho,-u,T}-\cM_{1,0,1})/\sqrt{\cM_{1,0,1}}\|_{L^\infty}\gtrsim |u|\|\,|\xi|\cM_{1,0,1}\|_{L^\infty}
$$
as $u\to 0$, and it might happen that $R_u<|u|\|\,|\xi|\cM_{1,0,1}\|_{L^\infty}$ for all $u\not=0$.

The main ingredient needed to understand the transition from evaporation to condensation in the context of the half-space problem (\ref{HalfSp-F}) is therefore to obtain for the operator $\cF_u$ --- and for the radius $R_u$ --- an estimate 
that is \textit{uniform in} $u$ as  $u\to 0$.  

The generalized eigenfunction $\phi_u$ is precisely the ingredient providing this uniform estimate, by a penalization algorithm described below.

\subsection{The Lyapunov-Schmidt Method}

We denote by $\Pi_+$ the $\fH$-orthogonal projection on $\bR X_+$, i.e.
\be\lb{DefPi+}
\Pi_+g=\la gX_+\ra X_+\,,
\ee
and likewise, by $\Pi$ the $\fH$-orthogonal projection on $\text{Span}\{X_+,X_0,X_-\}$, i.e.
\be\lb{DefPi}
\Pi g=\la gX_+\ra X_++\la gX_0\ra X_0+\la gX_-\ra X_-\,.
\ee

Moreover, we introduce, for all $u\in(-r,0)\cup(0,r)$ as in Proposition \ref{P-GEP}, the operators $\bp_u$ and $\bP_u$ defined by
\be\lb{Def-p}
\bp_ug=-\la(\xi_1+u)\psi_ug\ra\phi_u\,,\quad\bP_ug=-\la\psi_ug\ra(\xi_1+u)\phi_u\,,
\ee
where
\be\lb{Def-psi}
\psi_u:=\frac{\phi_u-\phi_0}{u}\,,\quad 0<|u|<r\,.
\ee
Since $u\mapsto\tau_u$ and $u\mapsto\phi_u$ are real-analytic on $(-r,r)$ with $\tau_0=0$, and since $\psi_u:=(\phi_u-\phi_0)/u$, the function $u\mapsto\psi_u$ is also real-analytic on $(-r,r)$ with values in $\Dom\cL$.

\begin{Lem}\lb{L-PtesPp}
The linear operators $\bp_u$ and $\bP_u$ are rank-$1$ projections defined on $\fH$, satisfying
$$
\bP_u((\xi_1+u)f)=(\xi_1+u)\bp_uf\,,\qquad f\in\fH\,,
$$
and
$$
\bP_u(\cL f)=\cL(\bp_uf)\,,\quad f\in\fH\cap\Dom\cL\text{ such that }(\xi_1+u)f\perp X_0\,.
$$ 
Besides
$$
(\xi_1+u)\phi_u\perp\Ker\cL\,,\quad\text{ and therefore }\IM\bP_u\subset\Ker\cL^\perp\,.
$$
\end{Lem}

\begin{proof}
The first property follows from a straightforward computation. 

For each $f\in\Dom\cL$, one has
$$
\ba
\bP_u(\cL f)=-(\xi_1+u)\phi_u\la\psi_u\cL f\ra=-(\xi_1+u)\phi_u\la f\cL\psi_u\ra
\\
=-(\xi_1+u)\phi_u\frac1u\la f\cL(\phi_0+u\psi_u)\ra=-(\xi_1+u)\phi_u\frac1u\la f\cL(\phi_u)\ra
\\
=-(\xi_1+u)\phi_u\frac1u\tau_u\la (\xi_1+u)\phi_uf\ra=-\frac1u\la (\xi_1+u)(\phi_0+u\psi_u)f\ra\cL\phi_u
\\
=-\la (\xi_1+u)\psi_uf\ra\cL\phi_u=\cL(\bp_uf)&\,,
\ea
$$
where the penultimate equality follows from assuming that $\la(\xi_1+u)fX_0\ra=0$.

Since $\tau_u\not=0$ whenever $0<|u|<r$, one has
$$
(\xi_1+u)\phi_u=\frac1{\tau_u}\cL\phi_u\in\IM\cL=(\Ker\cL)^\perp
$$
and this obviously entails the last property.

Finally, we check that $\bp_u$ and $\bP_u$ are projections:
$$
\ba
\bp_u^2(f)&=\la(\xi_1+u)\psi_u\phi_u\ra\la(\xi_1+u)\psi_uf\ra\phi_u=-\la(\xi_1+u)\psi_u\phi_u\ra\bp_u(f)\,,
\\
\bP_u^2(f)&=\la(\xi_1+u)\phi_u\psi_uf\ra\la\psi_uf\ra(\xi_1+u)\phi_u=-\la(\xi_1+u)\psi_u\phi_u\ra\bP_u(f)\,,
\ea
$$
and we conclude since
$$
\la(\xi_1+u)\psi_u\phi_u\ra=\frac1u(\la(\xi_1+u)\phi_u^2\ra-\la(\xi_1+u)\phi_0\phi_u\ra)=-1\,.
$$
Indeed, 
$$
\la(\xi_1+u)\phi_u^2\ra=-u\,,\quad\text{ and }\la(\xi_1+u)\phi_u\phi_0\ra=0\,,
$$
in view of the third property in the proposition, since $\phi_0\in\Ker\cL$.
\end{proof}

The projection $\bp_u$ is a deformation of the projection $\bp$ used in \cite{Golse2008} to study the half-space problem (\ref{HalfSp-F}) in the case $u=0$. The role of $\bp_u$ and $\bP_u$ is reminiscent of the Lyapunov-Schmidt method used by 
Nicolaenko-Thurber \cite{NicoThurber} to analyze the shock profile problem for the Boltzmann equation.

The following observations explain the origin of the penalization method used in the construction of the solution to (\ref{HalfSp-f}).

\begin{Lem}\lb{L-PtesHalfSp-f}
Assume that $0<|u|<r$. Let $Q$ satisfy (\ref{HypQ}) and $f$ be a solution to (\ref{HalfSpLin}) such that (\ref{Even-f}) holds. Assume that the source term satisfies 
$$
e^{\g x}Q\in L^\infty(\bR_+;\fH)\quad\text{ for some }\g>\max(\tau_u,0)\,, 
$$
and that 
$$
e^{\g x}f\in L^\infty(\bR_+;\fH)\,.
$$
Then 

\smallskip
\noindent
(a) the function $f$ satisfies
$$
\la(\xi_1+u)fX_+\ra=\la(\xi_1+u)fX_0\ra=\la(\xi_1+u)fX_-\ra=0\,,\quad x\ge 0\,;
$$
(b) one has
$$
(\xi_1+u)\bp_uf(x,\xi)=-\int_0^\infty e^{\tau_uz}\bP_uQ(x+z,\xi)dz\,,
$$
and, whenever $-r<u<0$,
\be
\lb{CondUnifDecay}
\la(\xi_1+u)\psi_uf\ra(0)+\int_0^\infty e^{\tau_uy}\la\psi_uQ\ra(y)dy=0\,.
\ee
\end{Lem}

\begin{proof}
Any solution of (\ref{HalfSpLin}) satisfies
$$
\ba
\d_x\la(\xi_1+u)X_\pm f\ra&=-\la X_\pm\cL f\ra+\la X_\pm Q\ra=0\,,
\\
\d_x\la(\xi_1+u)X_0 f\ra&=-\la X_0\cL f\ra\,+\,\la X_0 Q\ra\,=0\,.
\ea
$$
Besides,
$$
\ba
\la(\xi_1+u)X_\pm f\ra\to 0\text{ as }x\to+\infty\,,
\\
\la(\xi_1+u)X_0 f\ra\to 0\text{ as }x\to+\infty\,,
\ea
$$
so that statement (a) holds.

Now for (b). For $0<|u|<r$, applying the first and second identities in Lemma \ref{L-PtesPp} shows that
$$
(\xi_1+u)\d_x\bp_uf+\cL(\bp_uf)=\d_x\bP_u((\xi_1+u)f)+\bP_u\cL f=\bP_uQ\,,
$$
since $(\xi_1+u)f\perp\Ker\cL$. Besides
$$
\cL(\bp_uf)=-\la(\xi_1+u)\psi_uf\ra\cL\phi_u=-\la(\xi_1+u)\psi_uf\ra\tau_u(\xi_1+u)\phi_u=\tau_u(\xi_1+u)\bp_uf\,,
$$
so that
$$
(\xi_1+u)(\d_x\bp_uf+\tau_u\bp_uf)=\bP_uQ\,,
$$
or, equivalently
$$
\d_x\la(\xi_1+u)\psi_uf\ra+\tau_u\la(\xi_1+u)\psi_uf\ra=\la\psi_uQ\ra\,.
$$
For $u$ small enough, one has $\tau_u<\g$, so that
$$
\d_x\left(e^{\tau_u x}\la(\xi_1+u)\psi_uf\ra\right)=e^{\tau_u x}\la\psi_uQ\ra=O(e^{(\tau_u-\g)x})\,.
$$
At this point, we study separately the cases $u>0$ and $u<0$.

\medskip
\noindent
\textit{Step 1.} If $0<u<r$, then $\tau_u<0$, so that
$$
\ba
e^{\tau_u x}\la(\xi_1+u)\psi_uf\ra(x)&=-\int_x^\infty e^{\tau_u y}\la\psi_uQ\ra(y)dy
\\
&=-\int_0^\infty e^{\tau_u (x+z)}\la\psi_uQ\ra(x+z)dz
\ea
$$
i.e.
$$
\la(\xi_1+u)\psi_uf\ra(x)=-\int_0^\infty e^{\tau_uz}\la\psi_uQ\ra(x+z)dz\,.
$$
Then
$$
\ba
\left|\int_0^\infty e^{\tau_uz}\la\psi_uQ\ra(x+z,\cdot)dz\right|\le\|\psi_u\|_{\fH}\int_0^\infty e^{\tau_uz}\|Q(x+z)\|_\fH dz
\\
\le\|\psi_u\|_{\fH}\sup_{y>0}\left(e^{\g y}\|Q(y,\cdot)\|_{\fH}\right)\int_0^\infty e^{\tau_uz}e^{-\g(x+z)}dz&\,,
\ea
$$
so that
$$
|\la(\xi_1+u)\psi_uf\ra(x)|\le\|\psi_u\|_{\fH}\sup_{y>0}\left(e^{\g y}\|Q(y,\cdot)\|_{\fH}\right)\frac{e^{-\g x}}{\g-\tau_u}\,.
$$

\medskip
\noindent
\textit{Step 2.} If $-r<u<0$, then $\tau_u>0$, so that
$$
\ba
\la(\xi_1+u)\psi_uf\ra(x)=e^{-\tau_u x}\la(\xi_1+u)\psi_uf\ra(0)+\int_0^xe^{-\tau_u(x-y)}\la\psi_uQ\ra(y)dy
\\
=e^{-\tau_u x}\left(\la(\xi_1+u)\psi_uf\ra(0)+\int_0^\infty e^{\tau_uy}\la\psi_uQ\ra(y)dy\right)
\\
-e^{\tau_ux}\int_x^\infty e^{\tau_uy}\la\psi_uQ\ra(y)dy\,.
\ea
$$
Since
$$
\ba
e^{-\tau_ux}\left|\int_x^\infty e^{\tau_uy}\la\psi_uQ\ra(y)dy\right|=\left|\int_0^\infty e^{\tau_uz}\la\psi_uQ\ra(x+z)dz\right|
\\
\le\|\psi_u\|_\fH\sup_{y>0}(e^{\g y}\|Q(y,\cdot)\|_\fH)\int_0^\infty e^{\tau_uz}e^{-\g(x+z)}dz&\,,
\ea
$$
one has
$$
\ba
\left|\la(\xi_1+u)\psi_uf\ra(x)-e^{-\tau_u x}\left(\la(\xi_1+u)\psi_uf\ra(0)+\int_0^\infty e^{\tau_uy}\la\psi_uQ\ra(y)dy\right)\right|
\\
\le\|\psi_u\|_\fH\sup_{y>0}(e^{\g y}\|Q(y,\cdot)\|_\fH)\frac{e^{-\g x}}{\g-\tau_u}&\,.
\ea
$$

Therefore, if $-r<u<0$, in general 
$$
\la(\xi_1+u)\psi_uf\ra(x)=O(e^{-\tau_ux})\,.
$$
Since $\tau_u\sim\dot\tau_0u$ as $u\to 0$, this exponential decay is not uniform in $u$ near $u=0$, unless
$$
\la(\xi_1+u)\psi_uf\ra(0)+\int_0^\infty e^{\tau_uy}\la\psi_uQ\ra(y)dy=0\,,
$$
in which case 
$$
\la(\xi_1+u)\psi_uf\ra(x)=O(e^{-\g x})\,,
$$
and this is precisely statement (b) in Lemma \ref{L-PtesHalfSp-f}
\end{proof}

\smallskip
Thus we seek solutions $f$ of (\ref{HalfSpLin}) in the form 
\be
\lb{DecompPp}
f=g-h\phi_u\,,
\ee
where $g$ satisfies
\be
\lb{HalfSp-g}
\left\{
\ba
{}&(\xi_1+u)\d_xg+\cL g=(I-\bP_u)Q\,,\quad x>0\,,\,\,\xi\in\bR^3\,,
\\
&\la(\xi_1+u)\psi_ug\ra(x)=0\,,\quad x>0\,,
\\
&g(x,\cdot)\to 0\text{ in }\fH\text{ as }x\to+\infty\,,
\ea
\right.
\ee
while
\be
\lb{h=}
h(x)=-\int_0^\infty e^{\tau_uz}\la\psi_uQ\ra(x+z)dy\,.
\ee
Observe that the condition $\la(\xi_1+u)\psi_ug\ra=0$ is equivalent to the fact that $\bp_ug=0$, so that $g=(I-\bp_u)f$.

Notice that
$$
\ba
{}&\la(\xi_1+u)\psi_uf\ra(0)+\int_0^\infty e^{\tau_uy}\la\psi_uQ\ra(y)dy
\\
&\qquad=
\la(\xi_1+u)\psi_ug\ra(0)-h(0)\la(\xi_1+u)\psi_u\phi_u\ra+\int_0^\infty e^{\tau_uy}\la\psi_uQ\ra(y)dy=0\,,
\ea
$$
since $\la(\xi_1+u)\psi_ug\ra(0)=0$ and 
$$
\la(\xi_1+u)\psi_u\phi_u\ra=\la(\xi_1+u)(\tfrac1uX_0+\psi_u)\phi_u\ra=\tfrac1u\la(\xi_1+u)\phi_u^2\ra=-1\,.
$$
In other words, the function $f$ defined as in (\ref{DecompPp}) satisfies the uniform exponential decay condition (\ref{CondUnifDecay}).

\subsection{The Ukai-Yang-Yu Penalization Method}

The formulation (\ref{HalfSp-g}) is precisely the one for which we use a penalization method. Indeed, any solution $g\in L^\infty(\bR_+;\fH\cap\Dom\cL)$ to (\ref{HalfSp-g}) satisfies
$$
\d_x\Pi((\xi_1+u)g)=\Pi(I-\bP_u)Q=\Pi Q=0\,,\text{ since }\IM\cL+\IM\bP_u\subset(\Ker\Pi)^\perp\,.
$$
Since we have assumed that $g(x,\cdot)\to 0$ in $\fH$ as $x\to+\infty$, one has
$$
\Pi((\xi_1+u)g)=0\,.
$$
Likewise, 
$$
\bp_ug=0\,,\quad\text{ since }\la(\xi_1+u)\psi_ug\ra=0\,.
$$
Therefore, if $g\in L^\infty(\bR_+;\fH)$ is a solution of (\ref{HalfSp-g}), then 
\be\lb{g->gg}
g_\g(x,\xi):=e^{\g x}g(x,\xi)
\ee
is a solution of the penalized problem
\be
\lb{HalfSpPen}
\left\{
\ba
{}&(\xi_1+u)\d_xg_{u,\g}+\cL^p g_{u,\g}=e^{\g x}(I-P_u)Q\,,\quad x>0\,,\,\,\xi\in\bR^3\,,
\\
&g_{u,\g}\in L^\infty(\bR_+;\fH\cap\Dom\cL)\,,
\ea
\right.
\ee
where the penalized linearized collision operator is defined by
$$
\cL^p_ug:=\cL g+\a\Pi_+((\xi_1+u)g)+\b\bp_ug-\g(\xi_1+u)g\,,
$$
for all $\a,\b>0$.

Conversely, we should seek under which condition(s) a solution of the penalized problem (\ref{HalfSpPen}) with appropriately chosen $\a,\b$ defines a solution of the original problem (\ref{HalfSp-g}) via (\ref{g->gg}). This is explained in the 
next lemma.

\begin{Lem}\lb{L-PenCond}
For $0<|u|<r$, let 
$$
\cA_u=\left(\begin{matrix}\a &0 &-u\b\la\psi_uX_+\ra\\ 0 &0 &-\b\la\phi_uX_0\ra\\ \a\la\psi_uX_+\ra &\tfrac1u\tau_u &\tau_u-\b\la\psi_u\phi_u\ra\end{matrix}\right)\,.
$$
There exists $0<r'\le r$ such that, whenever $0<|u|<r'$, the matrix $\cA_u$ has 3 distinct eigenvalues 
$$
\l_1(u)>\l_2(u)>0>\l_3(u)\,,
$$
such that
$$
\inf_{0<|u|<r'}\l_2(u)>0>\sup_{0<|u|<r'}\l_3(u)\,.
$$
Let $(l_1(u),l_2(u),l_3(u))$ be a real-analytic basis of left eigenvectors of $\cA_u$ defined for $0<|u|<r'$, and set
$$
\ba
Y_1[u](\xi):=(X_+(\xi),X_0(\xi),\psi_u(\xi))\cdot l_1(u)\,,
\\
Y_2[u](\xi):=(X_+(\xi),X_0(\xi),\psi_u(\xi))\cdot l_2(u)\,.
\ea
$$
Then, if $g_{u,\g}$ satisfies (\ref{HalfSpPen}), one has
$$
\la(\xi_1+u)X_+g_{u,\g}\ra=\la(\xi_1+u)\psi_ug_{u,\g}\ra=0
\Leftrightarrow
\left\{\begin{array}{l}\la(\xi_1+u)Y_1[u]g_{u,\g}\ra\rstr_{x=0}=0\,,\\ \la(\xi_1+u)Y_2[u]g_{u,\g}\ra\rstr_{x=0}=0\,.\end{array}\right.
$$
\end{Lem}

\begin{proof}
A straightforward computation shows that
$$
\ba
{}&\d_x\la(\xi_1+u)X_+g_\g\ra+(\a-\g)\la(\xi_1+u)X_+g_\g\ra-u\b\la\psi_uX_+\ra\la(\xi_1+u)\psi_ug_\g\ra=0\,,
\\
&\d_x\la(\xi_1+u)X_0g_\g\ra-\g\la(\xi_1+u)X_0g_\g\ra-\b\la\phi_uX_0\ra\la(\xi_1+u)\psi_ug_\g\ra=0\,,
\\
&\d_x\la(\xi_1+u)\psi_ug_\g\ra+\tfrac1u\tau_u\la(\xi_1+u)X_0g_\g\ra+\tau_u\la(\xi_1+u)\psi_ug_\g\ra
\\
&\qquad+\a\la\psi_uX_+\ra\la(\xi_1+u)X_+g_\g\ra-(\g+\b\la\psi_u\phi_u\ra)\la(\xi_1+u)\psi_ug_\g\ra=0\,.
\ea
$$
Setting
$$
A_+=\la(\xi_1+u)X_+g_\g\ra\,,\quad A_0=\la(\xi_1+u)X_0g_\g\ra\,,\quad B=\la(\xi_1+u)\psi_ug_\g\ra\,,
$$
we see that
\be\lb{EqA+A0B}
\frac{d}{dx}\left(\begin{matrix}A_+\\ A_0\\ B\end{matrix}\right)+(\cA_u-\g I)\left(\begin{matrix}A_+\\ A_0\\ B\end{matrix}\right)=0\,.
\ee
Since the function $u\mapsto\psi_u$ is real-analytic on $(-r,r)$ with values in $\fH\cap\Dom\cL$, the matrix field $u\mapsto\cA_u$ is real-analytic on $(-r,r)$. Besides
$$
\cA_0=\left(\begin{matrix}\a &0 &0\\ 0 &0 &-\b\\ \a\la\psi_0X_+\ra &\dot\tau_0 &-\b\la\psi_0X_0\ra\end{matrix}\right)\,,
$$
with characteristic polynomial $(\a-\l)(\l^2-\b\la\psi_0X_0\ra\l+\dot\tau_0\b)$. Since $\b>0$ while $\dot\tau_0<0$, the matrix $\cA_0$ has $3$ simple eigenvalues, two of which, including $\a$, are positive, while one is negative. 

By a standard analytic perturbation argument, we therefore obtain 3 eigenvalues $\l_1(u),\l_2(u),\l_3(u)$ for $\cA_u$ that are real-analytic functions of $u$ defined on some neighborhood of the origin, and satisfy the inequalities mentioned in the 
statement of the lemma. The existence of left eigenvectors $l_1(u),l_2(u),l_3(u)$ of $\cA_u$ that are real-analytic functions of $u$ defined in some neighborhood of $0$ follows from the same argument --- see for instance chapter II, \S 1 in \cite{Kato}.

Choose $\g$ such that 
$$
0<\g<\inf_{0<|u|<r'}\l_2(u)\,.
$$
Taking the inner product of each side of (\ref{EqA+A0B}) with $l_1(u)$, $l_2(u)$ and $l_3(u)$, we see that
$$
(A_+,A_0,B)(x)\cdot l_j(u)e^{(\l_j(u)-\g)x}=\text{Const.}
$$
Since 
$$
\l_3(u)-\g<0\text{ and we have assumed that }g_\g\in L^\infty(\bR_+;\fH\cap\Dom\cL)\,,
$$
we conclude that $x\mapsto(A_+,A_0,B)(x)\cdot l_3(u)$ is bounded on $\bR_+$, so that
$$
(A_+,A_0,B)(x)\cdot l_3(u)e^{(\l_j(u)-\g)x}=0\,,\quad\text{ for all }x\in\bR_+\,.
$$
Therefore
$$
\ba
\Pi_+((\xi_1+u)g_\g)=0\text{ and }\bp_ug_\g=0
	&\Leftrightarrow\left\{\begin{array}{l}(A_+,A_0,B)\cdot l_1(u)=0\,,\\ (A_+,A_0,B)\cdot l_2(u)=0\,,\end{array}\right.
\\
&\Leftrightarrow\left\{\begin{array}{l}(A_+,A_0,B)\rstr_{x=0}\cdot l_1(u)=0\,,\\ (A_+,A_0,B)\rstr_{x=0}\cdot l_2(u)=0\,,\end{array}\right.
\ea
$$
in which case $g(x,\xi)=e^{-\g x}g_\g(x,\xi)$ is a solution of the original half-space problem (\ref{HalfSp-g}). Obviously, these conditions can be recast as in the statement of the lemma.
\end{proof}


\section{Resolution of the Penalized Linear Problem}


\subsection{The Penalized Linearized Collision Integral}
 
\begin{Prop}\lb{P-CoercL}
There exists $R>0$ defined in \eqref{Cond1}, $\nu^*>0$ defined in \eqref{DefNu*}, and $\Gamma$ defined in \eqref{DefGa}, such that 
$$
\la f\cL^p_uf\ra\ge\tfrac{\g}{24\nu^*}\la\nu f^2\ra\,,\quad\text{ for all } f\in\fH\cap\Dom\cL\,,
$$
for each $u$ such that $|u|\le R$, provided that 
\be\lb{Defabg}
\a=\b=2\g\quad\text{ and }\quad 0<\g\le\Gamma\,.
\ee
\end{Prop}

\begin{proof}
We first recall the Bardos-Caflisch-Nicolaenko weighted spectral gap estimate for $\cL$ (see equation (2.14) in \cite{BardCaflNico}): there exists $\ka_0>0$ such that
\be\lb{BCN>}
\la f\cL f\ra\ge\ka_0\la\nu(f-\Pi f)^2\ra\quad\text{ for each }f\in\fH\cap\Dom\cL\,.
\ee

Write
$$
w:=f-\Pi f\,,\qquad q:=\Pi f\,.
$$
Then
$$
\ba
\la f(\a\Pi_+((\xi_1+u)f)&+\b\bp_uf-\g(\xi_1+u)f)\ra
\\
=&\a\la qX_+\ra\la(\xi_1+u)X_+q\ra-\b\la q\phi_u\ra\la(\xi_1+u)\psi_uq\ra
\\
&+\a\la qX_+\ra\la(\xi_1+u)X_+w\ra-\b\la q\phi_u\ra\la(\xi_1+u)\psi_uw\ra
\\
&-\b\la w\phi_u\ra\la(\xi_1+u)\psi_uq\ra-\b\la w\phi_u\ra\la(\xi_1+u)\psi_uw\ra
\\
&-\g\la(\xi_1+u)q^2\ra-2\g\la(\xi_1+u)qw\ra-\g\la(\xi_1+u)w^2\ra\,.
\ea
$$
Since $(\xi_1+u)\phi_u\in\IM\cL=(\Ker\cL)^\bot$, one has
$$
\ba
\la(\xi_1+u)\psi_uq\ra=&\frac1u\left(\la(\xi_1+u)\phi_uq\ra-\la(\xi_1+u)X_0q\ra\right)
\\
=&-\frac1u\la(\xi_1+u)X_0q\ra=-\la X_0q\ra\,.
\ea
$$
Hence
$$
\la f(\a\Pi_+((\xi_1+u)f)+\b\bp_uf-\g(\xi_1+u)f)\ra=\cS_1[q]+\cS_2[q,w]+\cS_3[w]\,,
$$
with
$$
\ba
\cS_1[q]:=&\a\la qX_+\ra\la(\xi_1+u)X_+q\ra+\b\la X_0q\ra\la q\phi_u\ra
\\
&-\g\la(\xi_1+u)q^2\ra\,,
\\
\cS_2[q,w]:=&\a\la qX_+\ra\la(\xi_1+u)X_+w\ra+\b\la X_0q\ra\la w\phi_u\ra
\\
&-\b\la(\xi_1+u)\psi_uw\ra\la q\phi_u\ra-2\g\la(\xi_1+u)qw\ra\,,
\\
\cS_3[w]:=&-\b\la(\xi_1+u)\psi_uw\ra\la w\phi_u\ra-\g\la(\xi_1+u)w^2\ra\,.
\ea
$$
Note that
$$
\la(\xi_1+u)qw\ra^2\le\la|\xi_1+u|q^2\ra\la|\xi_1+u|w^2\ra\le\tfrac1{\nu_-^2}\la\nu q^2\ra\la\nu w^2\ra\,.
$$

Observe that
$$
\ba
\cS_1[q]\ge&(\a-\g)(c+u)q_+^2+(\b-u\g)q_0^2+\g(c-u)q_-^2
\\
&-\b|u|\|\psi_u\|_{L^2}|q_0|\sqrt{q_+^2+q_0^2+q_-^2}
\\
\ge&((\a-\g)(c+u)-\b|u|\|\psi_u\|_{L^2})q_+^2
\\
&+(\b-u\g-\b|u|\|\psi_u\|_{L^2})q_0^2
\\
&+(\g(c-u)-\b|u|\|\psi_u\|_{L^2})q_-^2\,,
\ea
$$
with
$$
q_\pm:=\la fX_\pm\ra\,,\qquad q_0:=\la fX_0\ra\,.
$$
In particular
$$
\cS_1[q]\ge\left(\min((\a-\g)(c+u),(\b-u\g),\g(c-u))-\b|u|\|\psi_u\|_{L^2}\right)(q_+^2+q_0^2+q_-^2)\,.
$$
Assume that
\be\lb{Cond1}
|u|\le R:=\min\left(\tfrac12r',c-1,\tfrac14\left(\sup_{|u|\le\min(r/2,1)}\|\psi_u\|_{L^2}\right)^{-1}\right)\,,
\ee
with $r'$ chosen as in Lemma \ref{L-PenCond}, and pick
\be\lb{Cond2}
\a=\b=2\g>0\,.
\ee

Then
$$
\ba
\cS_1[q]\ge&\left(\min(\g(c-|u|),\g(2-|u|))-\b|u|\|\psi_u\|_{L^2}\right)(q_+^2+q_0^2+q_-^2)
\\
\ge&\left(\g-\b|u|\|\psi_u\|_{L^2}\right)(q_+^2+q_0^2+q_-^2)
\\
\ge&\tfrac12\g(q_+^2+q_0^2+q_-^2)\ge\tfrac{\g}{6\nu^*}\la\nu q^2\ra
\ea
$$
since $0<c-1=\sqrt{\tfrac53}-1<1$, where
\be\lb{DefNu*}
\nu^*:=\max(\la\nu X_+^2\ra,\la\nu X_0^2\ra,,\la\nu X_-^2\ra)\,. 
\ee

On the other hand,
$$
|\cS_3[w]|\le\tfrac{\b\la\nu\psi_u^2\ra^{1/2}\la\phi_u^2/\nu\ra^{1/2}+\g}{\nu_-}\la\nu w^2\ra\,,
$$
so that, provided that $u$ satisfies \eqref{Cond1} while $\a,\b,\g$ satisfy \eqref{Cond2}, one has
$$
|\cS_3[w]|\le\tfrac12\ka_0\la\nu w^2\ra\,,
$$
if 
\be\lb{Cond3}
0<\g<\ka_0\nu_-\Bigg/\left(2+4\sup_{|u|\le R}\sqrt{\la\nu\psi_u^2\ra\la\phi_u^2/\nu\ra}\right)\,.
\ee

Finally
$$
|\cS_2[q,w]|\le\tfrac{2(\la\nu X_+^2\ra+\sup_{|u|\le R}\la\nu\phi_u^2\ra^{1/2}(\la\nu X_0^2\ra^{1/2}+\la\nu\psi_u^2\ra^{1/2})+1)}{\nu_-^2}\g\la\nu q^2\ra^{1/2}\la\nu w^2\ra^{1/2}\,.
$$

Thus, if $u$ satisfies \eqref{Cond1}, and $\a,\b,\g$ are chosen as in \eqref{Cond2}, \eqref{Cond3} and if
\be\lb{Cond4}
\g<\ka_0\nu_-^4\Big/48\nu^*(\la\nu X_+^2\ra+\sup_{|u|\le R}\la\nu\phi_u^2\ra^{1/2}(\la\nu X_0^2\ra^{1/2}+\la\nu\psi_u^2\ra^{1/2})+1)^2
\ee
one has
$$
\la f(\cL f+\a\Pi_+((\xi_1+u)f)+\b\bp_uf-\g(\xi_1+u)f)\ra\ge\tfrac{\ka_0}{4}\la\nu w^2\ra+\tfrac{\g}{12\nu^*}\la\nu q^2\ra\,.
$$
Therefore, the inequality in the proposition follows from the following choice of $\Gamma$:
\be\lb{DefGa}
\ba
{}&\Gamma:=\min(3\nu^*\ka_0,\Gamma_1,\Gamma_2)\,,\quad\quad\text{ with }\Gamma_1:=\tfrac{\ka_0\nu_-}{2+4\sup_{|u|\le R}\sqrt{\la\nu\psi_u^2\ra\la\phi_u^2/\nu\ra}}\,,
\\
&\Gamma_2:=\tfrac{\ka_0\nu_-^4}{48\nu^*(\la\nu X_+^2\ra+\sup_{|u|\le R}\la\nu\phi_u^2\ra^{1/2}(\la\nu X_0^2\ra^{1/2}+\la\nu\psi_u^2\ra^{1/2})+1)^2
}\,,
\ea
\ee
where $\ka_0$ is the Bardos-Caflisch-Nicolaenko spectral gap in \eqref{BCN>}. Obviously 
$$
\sup_{|u|\le R}(\|\nu\phi_u\|_\fH+\|\nu\psi_u\|_\fH)<\infty
$$
since the map $u\mapsto\psi_u$ is real-analytic on $(-r,r)$ with values in $\fH\cap\Dom\cL$.
\end{proof}

\subsection{The $L^2$ Theory}

Consider the unbounded operator defined on the Hilbert space $\cH=L^2(\bR_+;\fH)$ by
$$
\left\{
\ba
{}&\cT_uf=(\xi_1+u)\d_xf+\cL^pf\,,
\\
&\Dom\cT_u=\{\phi\in\cH\,|\,(\xi_1+u)\d_x\phi\text{ and }\nu\phi\in\cH\text{ while }\phi(0,\xi)=0\text{ for }\xi_1>-u\}\,,
\ea
\right.
$$
whose adjoint is
$$
\left\{
\ba
{}&\cT^*_ug=-(\xi_1+u)\d_xg+\cL f+\a(\xi_1+u)\Pi_+g+\b\bp_u^*g-\g(\xi_1+u)g\,,
\\
&\Dom\cT^*_u=\{\psi\in\cH\,|\,(\xi_1+u)\d_x\psi\text{ and }\nu\psi\in\cH\text{ while }\psi(0,\xi)=0\text{ for }\xi_1<-u\}\,,
\ea
\right.
$$
where
$$
\bp_u^*g=-(\xi_1+u)\psi_u\la\phi_u g\ra\,.
$$

Following the same argument as in the proof of Lemma 3.1 in \cite{Golse2008}, we arrive at the following statements.

\begin{Lem}\lb{L-InvT}
Let $R>0$ be defined by \eqref{Cond1}, and let $\a=\b=2\g>0$ satisfy
$$
0<\g\le\min(\Gamma,\tfrac12\nu_-)
$$
with $\Gamma$ defined in \eqref{DefGa}. Then there exists $\ka\equiv\ka(R,\nu_-,\g)>0$ such that
$$
\ba
{}&\ka\|\nu\phi\|_\cH\le\|\cT_u\phi\|_\cH\,,&&\quad\text{ for each }\phi\in\Dom\cT_u\,,
\\
&\ka\|\nu\psi\|_\cH\le\|\cT^*_u\psi\|_\cH\,,&&\quad\text{ for each }\psi\in\Dom\cT^*_u\,.
\ea
$$
uniformly in $|u|\le R$. In particular $\Ker\cT_u=\{0\}$ and $\IM\cT_u=\fH$ whenever $|u|\le R$.
\end{Lem}

\begin{proof}
We briefly recall the proof of Lemma 3.1 in \cite{Golse2008} for the sake of completeness. 

If $g\in\Dom\cT_u$, one has in particular 
$$
\nu g\in L^2(Md\xi dx)\quad\text{ and }\quad\la(\xi_1+u)g^2\ra\in C(\bR_+)\,,
$$
so that there exists $L_n\to\infty$ such that $\la(\xi_1+u)g^2\ra(L_n)\to 0$ as $n\to\infty$. Thus
$$
\ba
\int_0^{L_n}\la g\cT_ug\ra dx=&\tfrac12\la(\xi_1+u)g^2\ra(L_n)-\tfrac12\la(\xi_1+u)g^2\ra(0)+\int_0^{L_n}\la g\cL^pg\ra dx\,,
\ea
$$
and letting $n\to\infty$, one arrives at
$$
\ba
\|g\|_\cH\|\cT_ug\|_\cH\ge\int_0^{\infty}\la g\cT_ug\ra dx=&-\tfrac12\la(\xi_1+u)g^2\ra(0)+\int_0^\infty\la g\cL^pg\ra dx
\\
\ge&\int_0^\infty\la g\cL^pg\ra dx\ge\tfrac{\g}{24\nu^*}\|\sqrt\nu g\|_\cH^2\,.
\ea
$$
Notice that
$$
-\tfrac12\la(\xi_1+u)g^2\ra(0)\ge 0
$$
for $g\in\Dom\cT_u$ because of the boundary condition at $x=0$ included in the definition of the domain $\Dom\cT_u$. Hence
$$
\|\cT_ug\|_\cH\ge\tfrac{\g\nu_-^{1/2}}{24\nu^*}\|g\|_\cH\,.
$$

Next
$$
\ba
\|\cT_ug\|_\cH=&\|(\xi_1+u)\d_xg+(\nu-\g(\xi_1+u))g\|_\cH
\\
&-\|\cK g\|_\cH-\a\|\Pi_+((\xi_1+u)g)\|_\cH-\b\|\bp_ug\|_\cH
\\
\ge&\|(\xi_1+u)\d_xg+(\nu-\g(\xi_1+u))g\|_\cH
\\
&-(\a\la(\xi_1+u)^2X^2_+\ra^{1/2}+\b\la(\xi_1+u)^2\psi_u^2\ra^{1/2}\la\phi_u^2\ra^{1/2}+\|\cK\|)\|g\|_\cH
\ea
$$
so that
\be\lb{Coerc1}
\|(\xi_1+u)\d_xg+(\nu-\g(\xi_1+u))g\|_\cH\le C\|\cT_ug\|_\cH
\ee
with
\be\lb{DefC}
C\!:=\!\left(1\!+\!\tfrac{24\nu^*}{\g\nu_-^{1/2}}\sup_{|u|\le R}\left(\a\sqrt{\la(\xi_1\!+\!u)^2X^2_+\ra}\!+\!\b\sqrt{\la(\xi_1\!+\!u)^2\psi_u^2\ra\la\phi_u^2\ra}\!+\!\|\cK\|\right)\right)\,.
\ee

Given $S\in\cH$, solve for $h\in\Dom\cT_u$ the equation 
$$
(\xi_1+u)\d_xh+(\nu-\g(\xi_1+u))h=S\,,\quad x>0\,.
$$
Since $h\in\Dom\cT_u$ it satisfies the boundary condition $h(0,\xi)=0$ for $\xi_1>-u$, so that
$$
h(x,\xi)=\int_0^x\exp\left(-\left(\tfrac{\nu}{\xi_1+u}-\g\right)(x-y)\right)\frac{S(y,\xi)}{\xi_1+u}dy\,,\quad \xi_1+u>0\,,
$$
so that
$$
|h(x,\xi)|\le\int_0^x\exp\left(-\left(\tfrac{\nu}{\xi_1+u}-\g\right)(x-y)\right)\frac{|S(y,\xi)|}{|\xi_1+u|}dy\,,\quad \xi_1+u>0\,.
$$
On the other hand, since $h\in\cH$, there exists a sequence $x_n\to \infty$ such that $\la h(x_n,\cdot)^2\ra\to 0$, so that
$$
h(x,\xi)=\int_x^\infty\exp\left(-\left(\tfrac{\nu}{|\xi_1+u|}+\g\right)(y-x)\right)\frac{S(y,\xi)}{|\xi_1+u|}dy\,,\quad\xi_1+u<0\,,
$$
and hence
$$
\ba
|h(x,\xi)|\le&\int_x^\infty\exp\left(-\left(\tfrac{\nu}{|\xi_1+u|}+\g\right)(y-x)\right)\frac{|S(y,\xi)|}{|\xi_1+u|}dy
\\
\le&\int_x^\infty\exp\left(-\left(\tfrac{\nu}{|\xi_1+u|}-\g\right)(y-x)\right)\frac{|S(y,\xi)|}{|\xi_1+u|}dy\,,\quad\xi_1+u<0\,.
\ea
$$
Therefore
$$
|h(\cdot,\xi)|\le G(\cdot,\xi)\star(|S(\cdot,\xi)|\indc_{\bR_+})
$$
with
\be\lb{DefG}
G(z,\xi)=\frac{\indc_{z(\xi_1+u)>0}}{|\xi_1+u|}\exp\left(-\left(\tfrac{\nu}{|\xi_1+u|}-\g\right)|z|\right).
\ee

For future use, we compute, for all $p\ge 1$,
\be\lb{GLp}
\ba
\|G(\cdot,\xi)\|_{L^p}\le&\frac1{|\xi_1+u|}\frac1{(p(\tfrac{\nu}{|\xi_1+u|}-\g))^{1/p}}
\\
=&\frac1{p^{1/p}|\xi_1+u|^{1-1/p}(\nu-\g|\xi_1+u|)^{1/p}}
\ea
\ee

Then we conclude from Young's convolution inequality and \eqref{GLp} with $p=1$ that
$$
\ba
\|h(\cdot,\xi)\|_{L^2(\bR_+)}\le&\|G(\cdot,\xi)\|_{L^1}\|S(\cdot,\xi)\|_{L^2}
\\
\le&\frac{\|S(\cdot,\xi)\|_{L^2}}{\nu-\g|\xi_1+u|}\le\frac{\|S(\cdot,\xi)\|_{L^2}}{(1-\frac\g{\nu_-})\nu(\xi)}\,,
\ea
$$
for $|u|\le R$, and hence
$$
(1-\tfrac\g{\nu_-})\|\nu h\|_\cH\le\|S\|_\cH\,.
$$
Applying this to
$$
S:=(\xi_1+u)\d_xg+(\nu(\xi)-\g(\xi_1+u))g\,,
$$
and using the bound \eqref{Coerc1} shows that
$$
(1-\tfrac\g{\nu_-})\|\nu g\|_\cH\le C\|\cT_ug\|_\cH\,.
$$
This obviously implies the first inequality in the lemma with
\be\lb{DefKa}
\ka:=\frac{\nu_--\g}{C\nu_-}\quad\text{ with }C\text{ defined in \eqref{DefC}.}
\ee
The analogous inequality for the adjoint operator $\cT_u$ is obtained similarly.

Now the first inequality obviously implies that 
$$
\Ker\cT_u=\{0\}\quad\text{ for }|u|\le R\,.
$$
The second inequality implies that $\IM\cT_u=\cH$, according to Theorem 2.20 in \cite{BrezisFAEng}.
\end{proof}

\smallskip
A straightforward application of Lemma \ref{L-InvT} is the following existence and uniqueness result.

\begin{Prop}\lb{P-ExUnL2}
Let $R>0$ be defined by \eqref{Cond1}, set $\a=\b=2\g>0$ with $0<\g\le\min(\Gamma,\tfrac12\nu_-)$, where $\Gamma$ is defined in \eqref{DefGa}, and let $\ka\equiv\ka(R,\nu_-,\g)>0$ be given by \eqref{DefKa}. 

Let $Q$ satisfy $e^{\g x}Q\in\cH$, while $\nu g_b\in\fH$. Then, for each $|u|<R$, there exists a unique solution $g_{u,\g}\in\Dom\cT_u$ of the linearized penalized problem
\be
\lb{HalfSpLinPen}
\left\{
\ba
{}&(\xi_1+u)\d_xg_{u,\g}+\cL^p g_{u,\g}=e^{\g x}(I-\bP_u)Q\,,\quad x>0\,,\,\,\xi\in\bR^3\,,
\\
&g_{u,\g}(0,\xi)=g_b(\xi)\,,\quad\xi_1+u>0\,.
\ea
\right.
\ee
Moreover, this solution satisfies the estimate
\be\lb{cH}
\ba
\ka\|\nu g_{u,\g}\|_\cH\le&(1+\sup_{|u|\le R}\sqrt{\la\psi_u^2\ra\la(\xi_1+u)^2\phi_u^2\ra})\|e^{\g x}Q\|_\cH
\\
&+\left(\frac{\sqrt{2\g}}{\nu_-}+\frac{\|\cL\|_{B(\Dom\cL,\fH)}}{\nu_-\sqrt{2\g}}+2\g+\frac{2\g}{\nu_-}\sup_{|u|\le R}\sqrt{\la\psi_u^2\ra\la\phi_u^2\ra}\right)\|\nu g_b\|_\fH
\ea
\ee
uniformly in $|u|\le R$.
\end{Prop}

\begin{proof}
Set
$$
h(x,\xi)=g(x,\xi)-g_b(\xi)\indc_{\xi_1+u>0}e^{-\g x}\,,\quad x>0\,.
$$
Then $h\in\Dom\cT_u$ if and only if
$$
(\xi_1+u)\d_xg\in\cH\text{ and }\nu g\in\cH\,,\quad\text{ and }g(0,\xi)=g_b(\xi)\text{ for }\xi_1+u>0\,,
$$
in which case
$$
\ba
\cT_uh(x,\xi)=&\,\cT_ug(x,\xi)+\g e^{-\g x}(\xi_1+u)^+g_b(\xi)-e^{-\g x}\cL^p(g_b\indc_{\xi_1+u>0})(\xi)
\\
=&\,e^{\g x}(I-\bP_u)Q(x,\xi)+\g e^{-\g x}(\xi_1+u)^+g_b(\xi)-e^{-\g x}\cL^p(g_b\indc_{\xi_1+u>0})(\xi)
\\
=&:S(x,\xi)
\ea
$$
if and only if $g$ is a solution to the problem \eqref{HalfSpLinPen}. (We use systematically the classical notation $z^+=\max(z,0)$.) The right hand side is recast as
$$
\ba
S(x,\xi)=&e^{\g x}(I-\bP_u)Q(x,\xi)+2\g e^{-\g x}(\xi_1+u)^+g_b(\xi)-e^{-\g x}\cL(g_b\indc_{\xi_1+u>0})(\xi)
\\
&-\a e^{-\g x}\Pi_+((\xi_1+u)^+g_b)(\xi)-\b e^{-\g x}\bp_u\left(g_b\indc_{\xi_1+u>0}\right)(\xi)\,,
\ea
$$
and estimated as follows:
$$
\ba
\|S\|_\cH\le&(1+\sup_{|u|\le R}\sqrt{\la\psi_u^2\ra\la(\xi_1+u)^2\phi_u^2\ra})\|e^{\g x}Q\|_\cH
\\
&+\left(\frac{\sqrt{2\g}}{\nu_-}+\frac{\|\cL\|_{B(\Dom\cL,\fH)}}{\nu_-\sqrt{2\g}}+\frac{\a}{\sqrt{2\g}}+\frac{\b}{\nu_-\sqrt{2\g}}\sup_{|u|\le R}\sqrt{\la\psi_u^2\ra\la\phi_u^2\ra}\right)\|\nu g_b\|_\fH\,.
\ea
$$
One concludes with the first inequality in Lemma \ref{L-InvT}.
\end{proof}

\subsection{The $L^\infty$ Theory}

\subsubsection{From $\cH$ to $L^2(Md\xi;L^\infty(\bR_+))$}

We recall that the linearized collision operator $\cL$ is split as $\cL=\nu-\cK$, where $\cK$ is compact on $L^2(\bR^3;Md\xi )$ (Hilbert's decomposition). With the notation
$$
\tilde Q:=e^{\g x}(I-\bP_u)Q\,,
$$
the solution of (\ref{HalfSpLinPen}) in $\fH$ satisfies
$$
\ba
g_{u,\g}(x,\xi)=&\exp\left(-\left(\tfrac{\nu}{\xi_1+u}-\g\right)x\right)g_b(\xi)
\\
&+\int_0^x\exp\left(-\left(\tfrac{\nu}{\xi_1+u}-\g\right)(x-y)\right)\frac{(\cK^pg_{u,\g}+\tilde Q)(y,\xi)}{\xi_1+u}dy\,,&&\xi_1>-u\,,
\\
g_{u,\g}(x,\xi)=&\int_x^\infty\exp\left(-\left(\tfrac{\nu}{|\xi_1+u|}+\g\right)(y-x)\right)\frac{(\cK^pg_{u,\g}+\tilde Q)(y,\xi)}{|\xi_1+u|}dy\,,&&\xi_1<-u\,,
\ea
$$
where $\cK^p=\nu-\cL^p$. In particular
$$
\ba
|g_{u,\g}(x,\xi)|\le&|g_b(\xi)|
\\
&+\int_0^x\exp\left(-\left(\tfrac{\nu}{|\xi_1+u|}-\g\right)|x-y|)\right)\frac{|\cK^pg_{u,\g}+\tilde Q|(y,\xi)}{|\xi_1+u|}dy\,,&&\xi_1>-u\,,
\\
|g_{u,\g}(x,\xi)|\le&\int_x^\infty\exp\left(-\left(\tfrac{\nu}{|\xi_1+u|}-\g\right)|x-y|\right)\frac{|\cK^pg_{u,\g}+\tilde Q|(y,\xi)}{|\xi_1+u|}dy\,,&&\xi_1<-u\,.
\ea
$$
Hence
\be\lb{Ineq|g|}
|g_{u,\g}(\cdot,\xi)|\le|g_b(\xi)|+G\star|\cK^pg_{u,\g}|(\cdot,\xi)+G\star|\tilde Q|(\cdot,\xi)\,,
\ee
where the function $G$ has been defined in \eqref{DefG}. 

\begin{Lem}\lb{L-BoundG}
One has
$$
\ba
\|G(\cdot,\xi)\star\phi(\cdot,\xi)\|_{L^\infty(\bR_+)}\le&\frac{\|\phi(\cdot,\xi)\|_{L^\infty(\bR_+)}}{\nu(\xi)-\g|\xi_1+u|}\,,
\\
\|\indc_{|\xi_1+u|\ge 1}G\star\phi(\cdot,\xi)\|_{L^\infty(\bR_+)}\le&\frac{\|\phi(\cdot,\xi)\|_{L^2(\bR_+)}}{\sqrt{2\nu(\xi)-2\g|\xi_1+u|}}\,.
\ea
$$
Moreover, for each $\eps>0$, one has
$$
\ba
\|\indc_{|\xi_1+u|<1}G\star\phi(\cdot,\xi)\|_{L^\infty(\bR_+)}\le&\frac{\eps^{1/4}\indc_{|\xi_1+u|<1}\|\phi(\cdot,\xi)\|_{L^\infty(\bR_+)}}{(4/3)^{3/4}|\xi_1+u|^{1/4}(\nu(\xi)-\g|\xi_1+u|)^{3/4}}
\\
&+\frac1{2\sqrt{e\eps}}\frac{\indc_{|\xi_1+u|<1}\|\phi(\cdot,\xi)\|_{L^2(\bR_+)}}{\sqrt{2}(\nu(\xi)-\g|\xi_1+u|)}\,.
\ea
$$
\end{Lem}

\begin{proof}
The two first inequalities follow from Young's convolution inequality and the computation of the $L^p$ norms of $G$ in \eqref{GLp} with $p=1$ and $p=2$.

For each $\eps>0$, write
$$
G\star h(\cdot,\xi)=G_{1,\eps}\star h(\cdot,\xi)+G_{2,\eps}\star h(\cdot,\xi)
$$
were
$$
G_{1,\eps}(z,\xi)=G(z,\xi)\indc_{|z|<\eps}\text{ and }G_{2,\eps}(z,\xi)=G(z,\xi)\indc_{|z|\ge\eps}.
$$
Then
$$
\ba
\|G_{1,\eps}\star\phi(\cdot,\xi)\|_{L^\infty(\bR_+)}\le&\|\indc_{[0,\eps]}G(\cdot,\xi)\|_{L^1(\bR_+)}\|\phi(\cdot,\xi)\|_{L^\infty(\bR_+)}
\\
\le&\|\indc_{[0,\eps]}\|_{L^4(\bR_+)}\|G(\cdot,\xi)\|_{L^{4/3}(\bR_+)}\|\phi(\cdot,\xi)\|_{L^\infty(\bR_+)}
\\
=&\frac{\eps^{1/4}\|\phi(\cdot,\xi)\|_{L^\infty(\bR_+)}}{(4/3)^{3/4}|\xi_1+u|^{1/4}(\nu-\g|\xi_1+u|)^{3/4}}\,,
\ea
$$
while
$$
\ba
\|G_{2,\eps}\star\phi(\cdot,\xi)\|_{L^\infty(\bR_+)}\le&\|\indc_{[\eps,\infty)}G(\cdot,\xi)\|_{L^2(\bR_+)}\|\phi(\cdot,\xi)\|_{L^2(\bR_+)}
\\
\le&\frac{\exp\left(-\left(\tfrac{\nu}{|\xi_1+u|}-\g\right)\eps\right)\|\phi(\cdot,\xi)\|_{L^2(\bR_+)}}{\sqrt{2}|\xi_1+u|^{1/2}(\nu-\g|\xi_1+u|)^{1/2}}
\\
\le&\sqrt{\tfrac{\nu}{|\xi_1+u|}-\g}\exp\left(-\left(\tfrac{\nu}{|\xi_1+u|}-\g\right)\eps\right)\frac{\|\phi(\cdot,\xi)\|_{L^2(\bR_+)}}{\sqrt{2}(\nu-\g|\xi_1+u|)}
\\
\le&\frac1{\sqrt{2e\eps}}\frac{\|\phi(\cdot,\xi)\|_{L^2(\bR_+)}}{\sqrt{2}(\nu-\g|\xi_1+u|)}=\frac{\|\phi(\cdot,\xi)\|_{L^2(\bR_+)}}{2\sqrt{e\eps}(\nu-\g|\xi_1+u|)}\,.
\ea
$$
\end{proof}

\smallskip
Therefore, we deduce from \eqref{Ineq|g|} and Lemma \ref{L-BoundG} that
$$
\ba
\|g_{u,\g}(\cdot,\xi)\|_{L^\infty(\bR_+)}\le&|g_b(\xi)|+\frac{\|\tilde Q(\cdot,\xi)\|_{L^\infty(\bR_+)}}{\nu(\xi)-\g|\xi_1+u|}
\\
&+\frac{\|\cK^pg_{u,\g}(\cdot,\xi)\|_{L^2(\bR_+)}}{\sqrt{2\nu(\xi)-2\g|\xi_1+u|}}+\frac1{\sqrt{2e\eps}}\frac{\|\cK^pg_{u,\g}(\cdot,\xi)\|_{L^2(\bR_+)}}{\sqrt{2}(\nu(\xi)-\g|\xi_1+u|)}
\\
&+\frac{\eps^{1/4}\indc_{|\xi_1+u|\le 1}\|\cK^pg_{u,\g}(\cdot,\xi)\|_{L^\infty(\bR_+)}}{(4/3)^{3/4}|\xi_1+u|^{1/4}(\nu(\xi)-\g|\xi_1+u|)^{3/4}}\,.
\ea
$$

Denote by $k_j$ for $j=1,2,3$ the integral kernel of the operator $\cK_j$ in Lemma \ref{L-PropK}, and set 
$$
\ba
\tilde k(\xi,\zeta)=&k_1(\xi,\zeta)+k_2(\xi,\zeta)+k_3(\xi,\zeta)\ge 0\,,
\\
\tilde k^p(\xi,\zeta)=&\tilde k(\xi,\zeta)+\a|\zeta_1+u||X_+(\xi)||X_+(\zeta)|M(\zeta)+\b|\zeta_1+u||\phi_u(\xi)||\psi_u(\zeta)|M(\zeta)\,.
\ea
$$
Denote by $\tilde\cK^p$ the integral operator with kernel $\tilde k^p$:
$$
\tilde\cK^p\phi(\xi)=\int_{\bR^3}\tilde k^p(\xi,\zeta)\phi(\zeta)d\zeta\,,
$$
and set
$$
\cG_{u,\g}(\xi):=\|g(\cdot,\xi)\|_{L^\infty(\bR_+)}\,.
$$
Then, for each $\xi\in\bR^3$, one has
$$
|\cK^pg_{u,\g}(x,\xi)|\le\int_{\bR^3}\tilde k^p(\xi,\zeta)\|g_{u,\g}(\cdot,\zeta)\|_{L^\infty}d\zeta\quad\text{ for a.e. }x\ge 0\,,
$$
so that
$$
\|\cK^pg_{u,\g}(\cdot,\xi)\|_{L^\infty}\le\int_{\bR^3}|\tilde k^p(\xi,\zeta)|\|g_{u,\g}(\cdot,\zeta)\|_{L^\infty}d\zeta=\tilde\cK^p\cG_{u,\g}(\xi)\,.
$$
Hence, Lemma \ref{L-BoundG} and \eqref{Ineq|g|} imply that
\be\lb{IneqG}
\ba
\cG_{u,\g}(\xi)\le&|g_b(\xi)|+\frac{\|\tilde Q(\cdot,\xi)\|_{L^\infty}}{\nu(\xi)-\g|\xi_1+u|}
\\
&+\frac{\|\cK^pg_{u,\g}(\cdot,\xi)\|_{L^2(\bR_+)}}{\sqrt{2\nu(\xi)-2\g|\xi_1+u|}}+\frac1{\sqrt{2e\eps}}\frac{\|\cK^pg_{u,\g}(\cdot,\xi)\|_{L^2(\bR_+)}}{\sqrt{2}(\nu(\xi)-\g|\xi_1+u|)}
\\
&+\frac{\eps^{1/4}\indc_{|\xi_1+u|\le 1}\tilde\cK^p\cG_{u,\g}(\xi)}{(4/3)^{3/4}|\xi_1+u|^{1/4}(\nu(\xi)-\g|\xi_1+u|)^{3/4}}\,.
\ea
\ee

At this point, we use the following lemma.

\begin{Lem}\lb{L-ContKp}
For all $\a,\b\in\bR$ there exists $K_*[\a,\b]>0$, and for each $s\ge 0$, there exists $K_s\equiv K_s[\a,\b]>0$ such that
$$
\sup_{|u|\le\min(1,r/2)}\|\tilde\cK^p\|_{B(\fH)}\le K_*[\a,\b]\,,
$$
together with
$$
\sup_{|u|\le\min(1,r/2)}\|M^{1/2}\tilde\cK^p\|_{B(\fH,L^{\infty,1/2}(\bR^3))}\le K_{1/2}[\a,\b]\,,
$$
and
$$
\sup_{|u|\le\min(1,r/2)}\|M^{1/2}\tilde\cK^pM^{-1/2}\|_{B(L^{\infty,s}(\bR^3),L^{\infty,s+1}(\bR^3))}\le K_{s+1}[\a,\b]\,.
$$
\end{Lem}

\begin{proof}
Since
$$
\tilde\cK f-\tilde\cK^pf=\a\la|\xi_1+u||X_+|f\ra|X_+|+\b\la|\xi_1+u||\psi_u|f\ra|\phi_u|\,,
$$
one has
$$
\ba
\sup_{|u|\le\min(1,r/2)}&\|(\tilde\cK f-\tilde\cK^pf)\|_{\fH}
\\
&\le\tfrac{\a}{\nu_-}\|\nu X_+\|_\fH\|f\|_\fH+\tfrac{\b}{\nu_-}\sup_{|u|\le\min(1,r/2)}\|\nu\psi_u\|_\fH\|\phi_u\|_{\fH}\|f\|_\fH\,,
\ea
$$
and
$$
\ba
\sup_{|u|\le\min(1,r/2)}&\|(1+|\xi|)^s\sqrt{M}(\tilde\cK f-\tilde\cK^pf)(\xi)\|_{L^\infty(\bR^3)}
\\
&\le\tfrac{\a}{\nu_-}\|\nu X_+\|_\fH\|(1+|\xi|)^s\sqrt{M}X_+\|_{L^\infty(\bR^3)}\|f\|_\fH
\\
&+\tfrac{\b}{\nu_-}\sup_{|u|\le\min(1,r/2)}\|\nu\psi_u\|_\fH\|(1+|\xi|)^s\sqrt{M}\phi_u\|_{L^\infty(\bR^3)}\|f\|_\fH\,.
\ea
$$
Observe that
$$
\sup_{|u|\le\min(1,r/2)}\|\nu\psi_u\|_\fH<\infty\,,
$$
since the map $u\mapsto\psi_u$ is real-analytic on $(-r,r)$ with values in $\Dom\cL$. One concludes with Propositions \ref{P-PropK} and \ref{P-GEP}, using especially the bound
$$
\sup_{|u|\le\min(1,r/2)}\\|(1+|\xi|)^s\sqrt{M}\phi_u\|_{L^\infty(\bR^3)}\le C_s<\infty
$$
established in Proposition \ref{P-PropK}.
\end{proof}

Hence
$$
\ba
\left\|\frac{\indc_{|\xi_1+u|\le 1}\tilde\cK^p\cG_{u,\g}}{|\xi_1+u|^{1/4}(1+|\xi|)^{3/4}}\right\|^2_{\fH}\le&\int_{|\xi_1+u|\le 1}\frac{\|(1+|\zeta|)^{1/2}M^{1/2}\tilde\cK^p\cG_{u,\g}\|^2_{L^\infty(\bR^3)}d\xi}{|\xi_1+u|^{1/2}(1+|\xi|)^{5/2}}
\\
\le&J^2K_{1/2}[\a,\b]^2\|\cG_{u,\g}\|^2_{\fH}\,,
\ea
$$
with
$$
J:=\left(\int_0^1\frac{2d\zeta_1}{\sqrt{\zeta_1}}\int_{\bR^2}\frac{d\xi'}{(1+|\xi'|)^{5/2}}\right)^{1/2}<\infty\,.
$$

Henceforth, it will be convenient to use the notation
\be\lb{DefIds}
\cI_{d,s}:=\int_{\bR^d}\frac{d\zeta}{(1+|\zeta|)^s}\,.
\ee
Thus
$$
J^2=4I_{2,5/2}\,.
$$

Thus, Lemma \ref{L-ContKp} and \eqref{IneqG} imply that
$$
\ba
\|\cG_{u,\g}\|_{\fH}\le&\|g_b\|_{\fH}+\frac{\|\tilde Q\|_{L^2(Md\xi;L^\infty(\bR_+))}}{\nu_--\g}
\\
&+\frac{\|\cK^p\|_{B(\fH)}\|g_{u,\g}\|_{\cH}}{\sqrt{2(\nu_--\g)}}+\frac1{\sqrt{2e\eps}}\frac{\|\cK^p\|_{B(\fH)}\|g_{u,\g}\|_{\cH}}{\sqrt{2}(\nu_--\g)}
\\
&+\frac{\eps^{1/4}JK_{1/2}[\a,\b]}{(4/3)^{3/4}(\nu_--\g)^{3/4}}\|\cG_{u,\g}\|_{\fH}\,.
\ea
$$
Choosing
$$
\frac1{2\eps^{1/4}}=\frac{JK_{1/2}[\a,\b]}{(4/3)^{3/4}(\nu_--\g)^{3/4}}
$$
leads to the inequality
\be\lb{IneqNormg}
\ba
\tfrac12\|g_{u,\g}\|_{L^2(Md\xi;L^\infty(\bR_+))}\le&\|g_b\|_{L^2(Md\xi)}+\frac{\|\tilde Q\|_{L^2(Md\xi;L^\infty(\bR_+))}}{\nu_--\g}
\\
&+\frac{K_*[\a,\b]}{\sqrt{2(\nu_--\g)}}\left(1+\tfrac{3\sqrt{3}}{\sqrt{2e}}\frac{\cI_{2,5/2}K_{1/2}[\a,\b]^2}{(\nu_--\g)^2}\right)\|g_{u,\g}\|_{\cH}\,.
\ea
\ee

\subsubsection{From $L^2(Md\xi;L^\infty(\bR_+))$ to $L^\infty(\bR_+\times\bR^3)$}

We next return to the inequality \eqref{Ineq|g|}.  Obviously
$$
\|\cK^pg_{u,\g}(\cdot,\xi)\|_{L^\infty(\bR^3)}\le\tilde\cK^p\cG_{u,\g}(\xi)\,,\qquad\xi\in\bR^3\,.
$$
Then, the first inequality in Lemma \ref{L-BoundG} implies that
$$
\cG_{u,\g}(\xi)\le|g_b(\xi)|+\frac{\tilde\cK^p\cG_{u,\g}(\xi)}{\nu(\xi)-\g|\xi_1+u|}+\frac{\|\tilde Q(\cdot,\xi)\|_{L^\infty(\bR^3)}}{\nu(\xi)-\g|\xi_1+u|}\,.
$$
By the second inequality in Lemma \ref{L-ContKp}, one has
\be\lb{fH-Linfty}
\ba
\|(1+|\xi|)^{3/2}\sqrt{M}\cG_{u,\g}\|_{L^\infty(\bR^3)}\le&\|(1+|\xi|)^{3/2}\sqrt{M}g_b\|_{L^\infty(\bR^3)}
\\
&+\frac{\|(1+|\xi|)^{1/2}\sqrt{M}\tilde Q\|_{L^\infty(\bR_+\times\bR^3)}}{\nu_--\g}
\\
&+\frac{\|(1+|\xi|)^{1/2}\sqrt{M}\tilde\cK^p\cG_{u,\g}\|_{L^\infty(\bR^3)}}{\nu_--\g}
\\
\le&\|(1+|\xi|)^{3/2}\sqrt{M}g_b\|_{L^\infty(\bR^3)}
\\
&+\frac{\|(1+|\xi|)^{1/2}\sqrt{M}\tilde Q\|_{L^\infty(\bR_+\times\bR^3)}}{\nu_--\g}
\\
&+\frac{K_{1/2}[\a,\b]\|\cG_{u,\g}\|_{\fH}}{\nu_--\g}\,.
\ea
\ee
On the other hand, the third inequality in Lemma \ref{L-ContKp} implies that
\be\lb{s-s+1}
\ba
\|(1+|\xi|)^s\sqrt{M}\cG_{u,\g}\|_{L^\infty(\bR^3)}\le&\|(1+|\xi|)^s\sqrt{M}g_b\|_{L^\infty(\bR^3)}
\\
&+\frac{\|(1+|\xi|)^{s-1}\sqrt{M}\tilde Q\|_{L^\infty(\bR_+\times\bR^3)}}{\nu_--\g}
\\
&+\frac{\|(1+|\xi|)^{s-1}\sqrt{M}\tilde\cK^p\cG_{u,\g}\|_{L^\infty(\bR^3)}}{\nu_--\g}
\\
\le&\|(1+|\xi|)^s\sqrt{M}g_b\|_{L^\infty(\bR^3)}
\\
&+\frac{\|(1+|\xi|)^{s-1}\sqrt{M}\tilde Q\|_{L^\infty(\bR_+\times\bR^3)}}{\nu_--\g}
\\
&+\frac{K_{s-1}[\a,\b]\|(1+|\xi|)^{s-2}\sqrt{M}\cG_{u,\g}\|_{L^\infty(\bR^3)}}{\nu_--\g}
\ea
\ee
for each $s\ge 1$. 

Applying this inequality with $s=3$ and the previous inequality leads to the following statement, which summarizes our treatment of the penalized, linearized half-space problem. From the technical point of view, the proposition below is
the core of our analysis.

\begin{Prop}\lb{P-L8LinPen}
Let $R>0$ be given by \eqref{Cond1}, and $\a=\b=2\g$ with 
$$
0<\g\le\min(\Gamma,\tfrac12\nu_-)
$$
and $\Gamma$ defined by \eqref{DefGa}. Let $Q\in\cH$ and $g_b\in\fH$ satisfy
$$
(1+|\xi|)^3\sqrt{M}g_b\in L^{\infty}(\bR^3)\,,\quad\text{ and }\,\,\,e^{(\g+\de)x}(1+|\xi|)^2\sqrt{M}Q\in L^{\infty}(\bR_+\times\bR^3)\,,
$$
and
$$
Q(x,\cdot)\perp\Ker\cL\qquad\text{ for a.e. }x\ge 0\,.
$$
Then the solution $g_{u,\g}$ of (\ref{HalfSpLinPen}) (whose existence and uniqueness is established in Proposition \ref{P-ExUnL2}) satisfies the estimate
$$
\ba
\|(1+|\xi|)^3\sqrt{M}g_{u,\g}\|_{L^\infty(\bR_+\times\bR^3)}
\\
\le L\left(\|(1+|\xi|)^3\sqrt{M}g_b\|_{L^\infty(\bR^3)}+\|e^{(\g+\de)x}(1+|\xi|)^2\sqrt{M}Q\|_{L^\infty(\bR_+\times\bR^3)}\right)
\ea
$$
uniformly in $|u|\le R$, for some constant
$$
L\equiv L[\g,\nu_\pm,\de,R,K_*[2\g,2\g],K_{1/2}[2\g,2\g],K_2[2\g,2\g]]>0\,.
$$
\end{Prop}

\begin{proof}
Recall that 
$$
\ba
\tilde Q(x,\xi)=&e^{\g x}(I-\bP_u)Q(x,\xi)
\\
=&e^{\g x}Q(x,\xi)+e^{\g x}\la Q(x,\cdot)\psi_u\ra(\xi_1+u)\phi_u(\xi)\,.
\ea
$$
Hence
$$
\ba
\|(1+|\xi|)^2\sqrt{M}\tilde Q\|_{L^\infty(\bR_+\times\bR^3)}\le\|e^{\g x}(1+|\xi|)^2\sqrt{M}Q\|_{L^\infty(\bR_+\times\bR^3)}
\\
+\|(1+|\xi|)\sqrt{M}\phi_u\|_{L^\infty(\bR^3)}\|e^{\g x}(1+|\xi|)^2\sqrt{M}Q\|_{L^\infty(\bR_+\times\bR^3)}\int\frac{|\psi_u(\xi)|M^{1/2}d\xi}{(1+|\xi|)^2}
\\
\le\|e^{\g x}(1+|\xi|)^2\sqrt{M}Q\|_{L^\infty(\bR_+\times\bR^3)}\left(1+C_1\cI_{3,4}^{1/2}\sup_{|u|\le R}\|\nu^{1/2}\psi_u\|_\fH\right)&\,.
\ea
$$

By \eqref{fH-Linfty} and \eqref{s-s+1}
$$
\ba
\|(1+|\xi|)^3\sqrt{M}g_{u,\g}\|_{L^\infty_{x,\xi}}
\\
\le\left(1+\tfrac{K_2[\a,\b]}{\nu_--\g}\right)\left(\|(1+|\xi|)^3\sqrt{M}g_b\|_{L^\infty_{\xi}}+\tfrac{\|(1+|\xi|)^2\sqrt{M}\tilde Q\|_{L^\infty_{x,\xi}}}{\nu_--\g}\right)
\\
+\tfrac{K_2[\a,\b]K_{1/2}[\a,\b]}{(\nu_--\g)^2}\|g_{u,\g}\|_{L^2(Md\xi;L^\infty_x)}&\,.
\ea
$$
With \eqref{IneqNormg}, this inequality becomes
$$
\ba
\|(1+|\xi|)^3\sqrt{M}g_{u,\g}\|_{L^\infty_{x,\xi}}
\\
\le\!\left(1\!+\!\tfrac{K_2[\a,\b]}{\nu_--\g}\!+\!\tfrac{2\cI_{3,4}^{1/2}K_2[\a,\b]K_{1/2}[\a,\b]}{(\nu_--\g)^2}\right)\!\left(\|(1\!+\!|\xi|)^3\sqrt{M}g_b\|_{L^\infty_{\xi}}\!+\!\tfrac{\|(1+|\xi|)^2\sqrt{M}\tilde Q\|_{L^\infty_{x,\xi}}}{\nu_--\g}\right)
\\
+\tfrac{\sqrt{2}K_2[\a,\b]K_{1/2}[\a,\b]K_*[\a,\b]}{(\nu_--\g)^{5/2}}\left(1+\tfrac{3\sqrt{3}}{\sqrt{2e}}\frac{\cI_{2,5/2}K_{1/2}[\a,\b]^2}{(\nu_--\g)^2}\right)\|g_{u,\g}\|_{\cH}&\,.
\ea
$$
Next we inject in the right hand side of this inequality the bound on $\|g_{u,\g}\|_{\cH}$ obtained in \eqref{cH}, together with the bound for $\tilde Q$ obtained above. Since we have chosen $\a=\b=2\g$, one finds that
$$
\ba
\|(1+|\xi|)^3\sqrt{M}g_{u,\g}\|_{L^\infty_{x,\xi}}
\le\left(1\!+\!\tfrac{K_2[2\g,2\g]}{\nu_--\g}\!+\!\tfrac{2\cI_{3,4}^{1/2}K_2[2\g,2\g]K_{1/2}[2\g,2\g]}{(\nu_--\g)^2}\right)
\\
\times\left(\|(1\!+\!|\xi|)^3\sqrt{M}g_b\|_{L^\infty_{\xi}}\!+\!\tfrac{1+C_1\cI_{3,4}^{1/2}\|\nu^{1/2}\psi_u\|_\fH}{\nu_--\g}\|e^{\g x}(1+|\xi|)^2\sqrt{M}Q\|_{L^\infty_{x,\xi}}\right)
\\
+\tfrac{\sqrt{2}K_2[2\g,2\g]K_{1/2}[2\g,2\g]K_*[2\g,2\g]}{(\nu_--\g)^{5/2}}\left(1+\tfrac{3\sqrt{3}}{\sqrt{2e}}\tfrac{\cI_{2,5/2}K_{1/2}[2\g,2\g]^2}{(\nu_--\g)^2}\right)
\\
\times\left(\tfrac{1+\sqrt{\la\psi_u^2\ra\la(\xi_1+u)^2\phi_u^2\ra}}{\ka\nu_-}\sqrt{\tfrac{\cI_{3,4}}{2\de}}\|e^{(\g+\de)x}(1+|\xi|)^2\sqrt{M}Q\|_{L^\infty_{x,\xi}}\right.
\\
\left.+
\left(\tfrac{\sqrt{2\g}}{\ka\nu^2_-}+\tfrac{\|\cL\|_{B(\Dom\cL,\fH)}}{\ka\nu^2_-\sqrt{2\g}}+\tfrac{2\g(\nu_-+\sqrt{\la\psi_u^2\ra\la\phi_u^2\ra})}{\ka\nu^2_-}\right)\nu_+\cI_{3,4}^{1/2}\|(1\!+\!|\xi|)^3\sqrt{M}g_b\|_{L^\infty_{\xi}}\right)&\,,
\ea
$$
where $\ka$ is given by \eqref{DefKa}. This implies the announced estimate with $L$ given by
\be\lb{DefL}
\ba
L:=\sup_{|u|\le R}\max\left(\left(1\!+\!\tfrac{K_2[2\g,2\g]}{\nu_--\g}+\tfrac{2\cI_{3,4}^{1/2}K_2[2\g,2\g]K_{1/2}[2\g,2\g]}{(\nu_--\g)^2}\right)\right.
\\
+\tfrac{\sqrt{2}K_2[2\g,2\g]K_{1/2}[2\g,2\g]K_*[2\g,2\g]}{(\nu_--\g)^{5/2}}\left(1+\tfrac{3\sqrt{3}}{\sqrt{2e}}\tfrac{\cI_{2,5/2}K_{1/2}[2\g,2\g]^2}{(\nu_--\g)^2}\right)
\\
\times\left(\tfrac{\sqrt{2\g}}{\ka\nu^2_-}+\tfrac{\|\cL\|_{B(\Dom\cL,\fH)}}{\ka\nu^2_-\sqrt{2\g}}+\tfrac{2\g(\nu_-+\sqrt{\la\psi_u^2\ra\la\phi_u^2\ra})}{\ka\nu^2_-}\right)\nu_+\cI_{3,4}^{1/2},
\\
\left(1\!+\!\tfrac{K_2[2\g,2\g]}{\nu_--\g}\!+\!\tfrac{2\cI_{3,4}^{1/2}K_2[2\g,2\g]K_{1/2}[2\g,2\g]}{(\nu_--\g)^2}\right)\tfrac{1+C_1\cI_{3,4}^{1/2}\|\nu^{1/2}\psi_u\|_\fH}{\nu_--\g}
\\
+\tfrac{\sqrt{2}K_2[2\g,2\g]K_{1/2}[2\g,2\g]\|K_*[2\g,2\g]}{(\nu_--\g)^{5/2}}\left(1+\tfrac{3\sqrt{3}}{\sqrt{2e}}\tfrac{\cI_{2,5/2}K_{1/2}[2\g,2\g]^2}{(\nu_--\g)^2}\right)
\\
\left.\times\tfrac{1+\sqrt{\la\psi_u^2\ra\la(\xi_1+u)^2\phi_u^2\ra}}{\ka\nu_-}\sqrt{\tfrac{\cI_{3,4}}{2\de}}\right)&\,.
\ea
\ee
\end{proof}


\section{Solving the Nonlinear Problem}


\subsection{The Penalized Nonlinear Problem}

Given a boundary data $f_b\equiv f_b(\xi)$ satisfying the condition
$$
\sqrt{M}f_b\in L^{\infty,3}(\bR^3)\,,\qquad f_b\circ R=f_b\,,
$$
consider the following penalized, nonlinear half-space problem
\be
\lb{HalfSpNonlinPen}
\left\{
\ba
{}&(\xi_1+u)\d_xg_{u,\g}+\cL^p g_{u,\g}=e^{-\g x}(I-P_u)\cQ(g_{u,\g}-h_{u,\g}\phi_u,g_{u,\g}-h_{u,\g}\phi_u)\,,
\\
&h_{u,\g}(x)=-e^{-\g x}\int_0^\infty e^{(\tau_u-2\g)z}\la\psi_u\cQ(g_{u,\g}-h_{u,\g}\phi_u,g_{u,\g}-h_{u,\g}\phi_u)\ra(x+z)dz\,,
\\
&g_{u,\g}(0,\xi)=f_b(\xi)+h_{u,\g}(0)\phi_u(\xi)\,,\quad\xi_1+u>0\,.
\ea
\right.
\ee
In this section, we seek to solve the problem \eqref{HalfSpNonlinPen} by a fixed point argument assuming that the boundary data $f_b$ is small in $L^{\infty,3}(\bR^3)$.

\begin{Prop}\lb{P-ExUnNLPen}
There exists $\eps>0$ defined in \eqref{DefEps} such that, for each boundary data $f_b\equiv f_b(\xi)$ satisfying
$$
f_b\circ\cR=f_b\quad\text{ and }\,\,\|(1+|\xi|)^3\sqrt{M}f_b\|_{L^\infty(\bR^3)}\le\eps
$$
(with $\cR$ defined in \eqref{DefR}), the problem \eqref{HalfSpNonlinPen} has a unique solution $(g_{u,\g},h_{u,\g})$ satisfying the symmetry
$$
g_{u,\g}(x,\cR\xi)=g_{u,\g}(x,\xi)\quad\text{ for a.e. }\,\,(x,\xi)\in\bR_+\times\bR^3\,,
$$
and the estimate
$$
\|(1+|\xi|)^3\sqrt{M}g_{u,\g}\|_{L^\infty(\bR_+\times\bR^3)}+\|h_{u,\g}\|_{L^\infty(\bR_+)}\le 2L\eps
$$
where $L$ is given by \eqref{DefL}.
\end{Prop}

\medskip
We first recall a classical result on the twisted collision integral $\cQ$.

\begin{Prop}\lb{P-BoundQ}
For each $s\ge 1$, there exists $Q_s>0$ such that
$$
\|(1+|\xi|)^{s-1}\sqrt{M}\cQ(f,g)\|_{L^{\infty}(\bR^3)}\le Q_s\|(1+|\xi|)^s\sqrt{M}f\|_{L^{\infty}(\bR^3)}\|(1+|\xi|)^s\sqrt{M}g\|_{L^{\infty}(\bR^3)}
$$
for all $f,g\in L^{\infty,s}(\bR^3)$.
\end{Prop}

This inequality is due to Grad: see Lemma 7.2.6 in \cite{CerIllPul} for a proof.

\begin{proof}[Proof of Proposition \ref{P-ExUnNLPen}]
Set
$$
\ba
\cX:=\{(g,h)\text{ s.t. }(1+|\xi|)^3\sqrt{M}g\in L^\infty(\bR_+\times\bR^3)\,,\,\,h\in L^\infty(\bR_+)&
\\
\text{ and }g(x,\cR\xi)=g(x,\xi)\text{ for a.e. }(x,\xi)\in\bR_+\times\bR^3\}&\,,
\ea
$$
which is a Banach space for the norm
$$
\|(g,h)\|_\cX:=\|(1+|\xi|)^3\sqrt{M}g\|_{L^\infty(\bR_+\times\bR^3)}+\|h\|_{L^\infty(\bR_+)}\,.
$$

Given $(g_{u,\g},h_{u,\g})\in\cX$, solve for $(\tilde g_{u,\g},\tilde h_{u,\g})$ the half-space problem 
$$
\left\{
\ba
{}&(\xi_1+u)\d_x\tilde g_{u,\g}+\cL^p\tilde g_{u,\g}=e^{-\g x}(I-P_u)\cQ(g_{u,\g}-h_{u,\g}\phi_u,g_{u,\g}-h_{u,\g}\phi_u)\,,
\\
&\tilde h_{u,\g}(x)=-e^{-\g x}\int_0^\infty e^{(\tau_u-2\g)z}\la\psi_u\cQ(g_{u,\g}-h_{u,\g}\phi_u,g_{u,\g}-h_{u,\g}\phi_u)\ra(x+z)dz\,,
\\
&\tilde g_{u,\g}(0,\xi)=f_b(\xi)+\tilde h_{u,\g}(0)\phi_u(\xi)\,,\quad\xi_1+u>0\,,
\ea
\right.
$$
and call 
$$
\cS_{u,\g}:\,\cX\ni(g_{u,\g},h_{u,\g})\mapsto(\tilde g_{u,\g},\tilde h_{u,\g})\in\cX
$$
the solution map so defined.

Applying Proposition \ref{P-BoundQ} shows that
$$
\ba
\|(1+|\xi|)^2\sqrt{M}\cQ(g_{u,\g}-h_{u,\g}\phi_u,g_{u,\g}-h_{u,\g}\phi_u)\|_{L^{\infty}(\bR^3)}
\\
\le Q_3(\|(1+|\xi|)^3\sqrt{M}g_{u,\g}\|_{L^\infty(\bR_+\times\bR^3)}+\|h_{u,\g}\|_{L^\infty(\bR_+)}\|(1+|\xi|)^3\sqrt{M}\phi_u\|_{L^\infty(\bR^3)})^2
\\
\le Q_3(\|(1+|\xi|)^3\sqrt{M}g_{u,\g}\|_{L^\infty(\bR_+\times\bR^3)}+C_3\|h_{u,\g}\|_{L^\infty(\bR_+)})^2
\\
\le Q_3\max(1,C_3^2)\|(g_{u,\g},h_{u,\g})\|^2_\cX&\,.
\ea
$$
Hence
\be\lb{Bound-h}
\ba
\|\tilde h_{u,\g}\|_{L^\infty(\bR_+)}\le\tfrac{\cI_{3,4}^{1/2}\|\psi_u\|_{\fH}}{2\g-\tau_u}\|(1+|\xi|)^2\sqrt{M}\cQ(g_{u,\g}-h_{u,\g}\phi_u,g_{u,\g}-h_{u,\g}\phi_u)\|_{L^{\infty}(\bR^3)}
\\
\le\tfrac{\cI_{3,4}^{1/2}\|\psi_u\|_{\fH}}{2\g-\tau_u}Q_3\max(1,C_3^2)\|(g_{u,\g},h_{u,\g})\|^2_\cX&\,.
\ea
\ee

On the other hand, we apply Proposition \ref{P-L8LinPen} with $\de=\g$ and
$$
g_b:=f_b+\tilde h_{u,\g}(0)\phi_u\,,\qquad Q:=e^{-2\g x}\cQ(g_\g-h_\g\phi_u,g_\g-h_\g\phi_u)\,.
$$
Then
$$
\ba
\|(1+|\xi|)^3\sqrt{M}g_b\|_{L^\infty(\bR^3)}\le&\|(1+|\xi|)^3\sqrt{M}f_b\|_{L^\infty(\bR^3)}
\\
&+\tfrac{\cI_{3,4}^{1/2}\|\psi_u\|_{\fH}}{2\g-\tau_u}Q_3\max(1,C_3^2)C_3\|(g_{u,\g},h_{u,\g})\|^2_\cX\,,
\ea
$$
while
$$
\ba
\|e^{(\g+\de)x}(1+|\xi|)^2\sqrt{M}Q\|_{L^\infty(\bR_+\times\bR^3)}
\\
=\|(1+|\xi|)^2\sqrt{M}\cQ(g_{u,\g}-h_{u,\g}\phi_u,g_{u,\g}-h_{u,\g}\phi_u)\|_{L^{\infty}(\bR_+\times\bR^3)}
\\
\le Q_3\max(1,C_3^2)\|(g_{u,\g},h_{u,\g})\|^2_\cX&\,.
\ea
$$
The bound on the solution to the penalized linearized problem in Proposition \ref{P-L8LinPen}, together with the estimate \eqref{Bound-h} for $\|h_{u,\g}\|_{L^\infty(\bR_+)}$, implies that
$$
\ba
\|(\tilde g_{u,\g},\tilde h_{u,\g})\|_\cX\le&L\|(1+|\xi|)^3\sqrt{M}f_b\|_{L^\infty(\bR^3)}
\\
&+ L\tfrac{\cI_{3,4}^{1/2}\|\psi_u\|_{\fH}}{2\g-\tau_u}Q_3\max(1,C_3^2)C_3\|(g_{u,\g},h_{u,\g})\|^2_\cX
\\
&+LQ_3\max(1,C_3^2)\|(g_{u,\g},h_{u,\g})\|^2_\cX
\\
&+\tfrac{\cI_{3,4}^{1/2}\|\psi_u\|_{\fH}}{2\g-\tau_u}Q_3\max(1,C_3^2)\|(g_{u,\g},h_{u,\g})\|^2_\cX\,,
\ea
$$
which we put in the form
$$
\|(\tilde g_{u,\g},\tilde h_{u,\g})\|_\cX\le L\|(1+|\xi|)^3\sqrt{M}f_b\|_{L^\infty(\bR^3)}+\Lambda\|(g_{u,\g},h_{u,\g})\|^2_\cX
$$
with
\be\lb{DefLa}
\Lambda:=Q_3\max(1,C_3^2)\left(\tfrac{\cI_{3,4}^{1/2}\|\psi_u\|_{\fH}}{2\g-\tau_u}(1+LC_3)+L\right)\,.
\ee

Pick $\eps>0$ small enough so that
\be\lb{DefEps}
0<\eps<1/4\Lambda L
\ee
and assume that
$$
\|(1+|\xi|)^3\sqrt{M}f_b\|_{L^\infty(\bR^3)}<\eps\,.
$$
If $\|(g_{u,\g},h_{u,\g})\|_\cX\le 2L\eps$, one has
$$
\|(\tilde g_{u,\g},\tilde h_{u,\g})\|_\cX\le L\eps+\Lambda(2L\eps)^2=L\eps+4\Lambda L^2\eps^2\le 2L\eps\,,
$$
so that the solution map $\cS_{u,\g}$ satisfies
$$
\cS_{u,\g}(\overline{B_\cX(0,2L\eps)})\subset\overline{B_\cX(0,2L\eps)}\,.
$$

Let  $(g_{u,\g},h_{u,\g})$ and $(g'_{u,\g},h'_{u,\g})\in\overline{B_\cX(0,2L\eps)}$. We seek to bound
$$
\|(1+|\xi|)^3\sqrt{M}(\cT(g_{u,\g},h_{u,\g})-\cT(g'_{u,\g},h'_{u,\g}))\|
$$
in terms of 
$$
\|(1+|\xi|)^3\sqrt{M}(g_{u,\g}-g'_{u,\g},h_{u,\g}-h'_{u,\g})\|_\cX\,.
$$
One has
$$
\left\{
\ba
{}&(\xi_1+u)\d_x(g_{u,\g}-g'_{u,\g})+\cL^p(g_{u,\g}-g'_{u,\g})=e^{-\g x}(I\!-\!P_u)\Sigma\,,
\\
&\Sigma=\cQ((g_{u,\g}\!-\!g'_{u,\g})\!-\!(h_{u,\g}\!-\!h'_{u,\g})\phi_u,(g_{u,\g}\!+\!g'_{u,\g})\!-\!(h_{u,\g}\!+\!h'_{u,\g})\phi_u)\,,
\\
&(\tilde h_{u,\g}-\tilde h'_{u,\g})(x)=-e^{-\g x}\int_0^\infty e^{(\tau_u-2\g)z}\la\psi_u\Sigma\ra(x+z)dz\,,
\\
&(\tilde g_{u,\g}-\tilde g'_{u,\g})(0,\xi)=(\tilde h_{u,\g}-\tilde h'_{u,\g})(0)\phi_u(\xi)\,,\quad\xi_1+u>0\,.
\ea
\right.
$$

First, we deduce from Proposition \ref{P-BoundQ} that
$$
\ba
\|(1+|\xi|)^2\sqrt{M}\Sigma\|_{L^{\infty}(\bR^3)}
\\
\le Q_3(\|(1+|\xi|)^3\sqrt{M}(g_{u,\g}-g'_{u,\g})\|_{L^\infty(\bR_+\times\bR^3)}+C_3\|h_{u,\g}-h'_{u,\g}\|_{L^\infty(\bR_+\times\bR^3)})
\\
\times(\|(1+|\xi|)^3\sqrt{M}(g_{u,\g}+g'_{u,\g})\|_{L^\infty(\bR_+\times\bR^3)}+C_3\|h_{u,\g}+h'_{u,\g}\|_{L^\infty(\bR_+)})
\\
\le Q_3\max(1,C_3^2)\|(g_{u,\g}-g'_{u,\g},h_{u,\g}-h'_{u,\g})\|_\cX\|(g_{u,\g}+g'_{u,\g},h_{u,\g}+h'_{u,\g})\|_\cX
\\
\le  4L\eps Q_3\max(1,C_3^2)\|(g_{u,\g}-g'_{u,\g},h_{u,\g}-h'_{u,\g})\|_\cX&\,.
\ea
$$

With this estimate, we bound $\tilde h_{u,\g}-\tilde h_{u,\g}$ as follows:
$$
\ba
\|\tilde h_{u,\g}-\tilde h_{u,\g}\|_{L^\infty(\bR_+)}\le\tfrac{\cI_{3,4}^{1/2}\|\psi_u\|_{\fH}}{2\g-\tau_u}\|(1+|\xi|)^2\sqrt{M}\Sigma\|_{L^{\infty}(\bR^3)}
\\
\le 4L\eps Q_3\max(1,C_3^2)\tfrac{\cI_{3,4}^{1/2}\|\psi_u\|_{\fH}}{2\g-\tau_u}\|(g_{u,\g}-g'_{u,\g},h_{u,\g}-h'_{u,\g})\|_\cX&\,.
\ea
$$

Finally, we apply Proposition \ref{P-L8LinPen} with $\de=\g$ and
$$
g_b:=(\tilde h_{u,\g}-\tilde h'_{u,\g})(0)\phi_u\,,\qquad Q:=e^{-2\g x}\Sigma\,.
$$
Thus
$$
\ba
\|(1+|\xi|)^3\sqrt{M}g_b\|_{L^\infty(\bR^3)}
\\
\le 4L\eps \tfrac{\cI_{3,4}^{1/2}\|\psi_u\|_{\fH}}{2\g-\tau_u}Q_3\max(1,C_3^2)C_3\|(g_{u,\g}-g'_{u,\g},h_{u,\g}-h'_{u,\g})\|_\cX&\,,
\ea
$$
while
$$
\ba
\|e^{(\g+\de)x}(1+|\xi|)^2\sqrt{M}Q\|_{L^\infty(\bR_+\times\bR^3)}=\|(1+|\xi|)^2\sqrt{M}\Sigma\|_{L^{\infty}(\bR_+\times\bR^3)}
\\
\le 4L\eps Q_3\max(1,C_3^2)\|(g_{u,\g}-g'_{u,\g},h_{u,\g}-h'_{u,\g})\|_\cX&\,.
\ea
$$
Hence, the bound in Proposition \ref{P-L8LinPen} implies that
$$
\ba
\|(\tilde g_{u,\g}-\tilde g'_{u,\g},\tilde h_{u,\g}-\tilde h'_{u,\g})\|_\cX
\\
\le4L^2\eps\tfrac{\cI_{3,4}^{1/2}\|\psi_u\|_{\fH}}{2\g-\tau_u}Q_3\max(1,C_3^2)C_3\|(g_{u,\g}-g'_{u,\g},h_{u,\g}-h'_{u,\g})\|_\cX
\\
+4L^2\eps Q_3\max(1,C_3^2)\|(g_{u,\g}-g'_{u,\g},h_{u,\g}-h'_{u,\g})\|_\cX
\\
+4L\eps\tfrac{\cI_{3,4}^{1/2}\|\psi_u\|_{\fH}}{2\g-\tau_u}Q_3\max(1,C_3^2)\|(g_{u,\g}-g'_{u,\g},h_{u,\g}-h'_{u,\g})\|_\cX
\\
\le 4L\eps Q_3\max(1,C_3^2)\left(L+\tfrac{\cI_{3,4}^{1/2}\|\psi_u\|_{\fH}}{2\g-\tau_u}(1+LC_3)\right)\|(g_{u,\g}-g'_{u,\g},h_{u,\g}-h'_{u,\g})\|_\cX
\\
= 4L\eps \Lambda\|(g_{u,\g}-g'_{u,\g},h_{u,\g}-h'_{u,\g})\|_\cX&\,.
\ea
$$

The inequality above implies that the solution map $\cS_{u,\g}$ satisfies
$$
\|\cS_{u,\g}(\tilde g_{u,\g},\tilde h_{u,\g})-\cS_{u,\g}(\tilde g'_{u,\g},\tilde h'_{u,\g})\|_\cX\le 4L\eps \Lambda\|(g_{u,\g}-g'_{u,\g},h_{u,\g}-h'_{u,\g})\|_\cX
$$
for all $(g_{u,\g},h_{u,\g})$ and $(g'_{u,\g},h'_{u,\g})\in\overline{B_\cX(0,2L\eps)}$. Since $0<\eps<1/4L\Lambda$, this implies that $\cS_{u,\g}$ is a strict contraction on $\overline{B_\cX(0,2L\eps)}$, which is a complete metric space 
(as a closed subset of the Banach space $\cX$). By the fixed point theorem, we conclude that there exists a unique $(g_{u,\g},h_{u,\g})\in\overline{B_\cX(0,2L\eps)}$ such that
$$
\cS_{u,\g}(g_{u,\g},h_{u,\g})=(g_{u,\g},h_{u,\g})\,.
$$
In other words, there exists a unique $(g_{u,\g},h_{u,\g})$ which is a solution of the problem \eqref{HalfSpNonlinPen} in $\overline{B_\cX(0,2L\eps)}$.
\end{proof}

\subsection{Removing the Penalization}

Let  $f_b\equiv f_b(\xi)$ satisfy
$$
f_b\circ\cR=f_b\quad\text{ and }\,\,\|(1+|\xi|)^3\sqrt{M}f_b\|_{L^\infty(\bR^3)}\le\eps
$$
(with $\cR$ as in \eqref{DefR}), and let $(g_{u,\g},h_{u,\g})$ be the unique solution to \eqref{HalfSpNonlinPen} given by Proposition \ref{P-ExUnNLPen}. Define
\be\lb{DefRug}
\fR_{u,\g}[f_b](\xi):=g_{u,\g}(0,\xi)\,,\quad\xi\in\bR^3\,.
\ee
By Lemma \ref{L-PenCond},
$$
\ba
\la(\xi_1+u)Y_1[u]\fR_{u,\g}[f_b]\ra=\la(\xi_1+u)Y_2[u]\fR_{u,\g}[f_b]\ra=0
\\
\iff\la(\xi_1+u)X_+g_{u,\g}\ra=\la(\xi_1+u)\psi_ug_{u,\g}\ra=0
\\
\implies\cL^pg_{u,\g}=\cL g_{u,\g}-\g(\xi_1+u)g_{u,\g}&\,.
\ea
$$

In that case, denoting
$$
g_u(x,\xi):=e^{-\g x}g_{u,\g}(x,\xi)\,,\quad h_u(x):=e^{-\g x}h_{u,\g}(x)\,,
$$
and
$$
f_u(x,\xi):=g_u(x,\xi)-h_u(x)\phi_u(\xi)\,,
$$
we see that 
$$
(\xi_1+u)\d_xg_u+\cL g_u=(I-\bP_u)\cQ(f_u,f_u)\quad \text{and }\,\, g_u=(I-\bp_u)f_u\,,
$$
while
$$
h_u(x)=-\int_0^\infty e^{\tau_uz}\la\psi_u\cQ(f_u,f_u)\ra(x+z)dz\quad\hbox{ and }h_u(x)\phi_u(\xi)=\bp_uf_u(x,\xi)\,.
$$
In other words, $f_u$ satisfies \eqref{NLHalfSpPbm} together with the bound \eqref{ExpDecUnif} with
\be\lb{DefE}
E:=2L\eps\max(1,C_3)\,.
\ee

This estimate holds for all $u$ satisfying $|u|\le R$ where $R$ is defined in \eqref{Cond1}, and all $\g\in(0,\min(\Gamma,\tfrac12\nu_-))$, where $\Gamma$ is defined in \eqref{DefGa}. The constants $L$ and $\eps$ are defined 
in \eqref{DefL} and \eqref{DefEps} respectively.

Conversely, if $f_u$ is a solution to the nonlinear half-space problem satisfying the uniform exponential decay condition \eqref{ExpDecUnif} for all $u$ satisfying $|u|\le R$ and all $\g\in(0,\min(\Gamma,\tfrac12\nu_-))$, we deduce from
Lemma \ref{L-PtesHalfSp-f} that $g_u:=(I-\bp_u)f_u$ and $h_u:=\la(\xi_1+u)\psi_uf_u\ra$ satisfy \eqref{HalfSp-g} and \eqref{h=} respectively, with $Q=\cQ(f,f)$, while $(g_{u,\g},h_{u,\g})$ defined by the formulas
\be\lb{FormPenSol}
\ba
g_{u,\g}(x,\xi)&=e^{\g x}g_u(x,\xi)=e^{\g x}(I-\bp_u)f_u(x,\xi)
\\
h_{u,\g}&=e^{\g x}h_u(x)=e^{\g x}\la(\xi_1+u)\psi_uf_u\ra(x)
\ea
\ee
must satisfy the penalized nonlinear half-space problem \eqref{HalfSpNonlinPen}. And since $(g_{u,\g},h_{u,\g})$ is a solution of \eqref{HalfSpNonlinPen} of the form \eqref{FormPenSol} with $f_u(x,\xi)\to 0$ in $\fH$ as $x\to\infty$,
one has
$$
\la(\xi_1+u)X_+g_{u,\g}\ra(x)=e^{\g x}\la(\xi_1+u)X_+g_u\ra(x)=0
$$
(because $\la(\xi_1+u)X_+g_u\ra$ is constant and $g_u(x,\cdot)\to 0$ in $\fH$ as $x\to+\infty$), and
$$
\la(\xi_1+u)\psi_ug_{u,\g}\ra(x)=e^{\g x}\la(\xi_1+u)\psi_ug_u\ra(x)=0
$$
(because $g_{u,\g}(x,\xi)=e^{\g x}(I-\bp_u)f_u(x,\xi)$). Applying Lemma \ref{L-PenCond} shows that $g_{u,\g}$ must therefore satisfy the conditions
$$
\la(\xi_1+u)Y_1[u]g_{u,\g}\ra(0)=\la(\xi_1+u)Y_2[u]g_{u,\g}\ra(0)=0\,,
$$
or in other words, 
$$
\la(\xi_1+u)Y_1[u]\fR_{u,\g}[f_b]\ra=\la(\xi_1+u)Y_2[u]\fR_{u,\g}[f_b]\ra=0\,.
$$

\bigskip
\noindent
\textbf{Acknowledgements.} This paper has been a long standing project, and we are grateful to several colleagues, whose works have been a source of inspiration and encouragement to us. Profs. Y. Sone and K. Aoki patiently explained to us the very intricate details 
of rarefied gas flows with evaporation and condensation in a half-space on the basis of their pioneering numerical work on the subject. The corresponding theoretical study in \cite{LiuYu} was presented in a course given by Prof. T.-P. Liu at Academia Sinica in Taipei in 
2009. The second author learned from the penalization method used here from a talk by the late Prof. S. Ukai in Stanford in 2001, and was introduced to the generalized eigenvalue problem by the late Prof. B. Nicolaenko in Phoenix in 1999. Niclas Bernhoff's visits to 
Paris, where most of this collaboration took place, were supported by the French Institute in Sweden (FR\"O program) in September 2016, and by SVeFUM in October 2017 and May 2019. Fran\c cois Golse also acknowledges Prof. S. Takata's hospitality in 2018 in 
Kyoto University, where this project was finalized.



\end{document}